\documentclass[11pt,a4paper,twoside]{amsart}
\usepackage[UKenglish]{babel}
\usepackage[T1]{fontenc}
\usepackage[utf8x]{inputenc}
\usepackage{latexsym,amsfonts,amsmath,amsthm,amssymb,mathrsfs}
\usepackage{fullpage}
\usepackage{xcolor}
\usepackage{paralist}
\usepackage{todonotes}
\usepackage{dsfont}
\usepackage{float}
\usepackage{caption}
\usepackage{xfrac}
\usepackage{hyperref}
\usepackage{ulem}
\usepackage{graphicx}
\usepackage{array}
\usepackage{mathrsfs}
\usepackage{fourier}
\usepackage{tikz} 
\usetikzlibrary{arrows.meta,decorations.pathreplacing,patterns} 
\definecolor{acol}{HTML}{246B91} 
\definecolor{bcol}{HTML}{256B32} 
\definecolor{ccol}{HTML}{C27323}

\DeclareMathOperator{\R}{\mathbb{R}}
\DeclareMathOperator{\F}{\mathcal{F}}
\DeclareMathOperator{\N}{\mathbb{N}}

\DeclareMathOperator{\T}{\mathbb{T}}
\DeclareMathOperator{\E}{\mathbb{E}}

\DeclareMathOperator{\sign}{sign}

\newtheorem{thm}{Theorem}

\newtheorem{lem}{Lemma}[section]
\newtheorem{cor}[thm]{Corollary}

\newtheorem{prop}[lem]{Proposition}
\newtheorem{rem}[lem]{Remark}
\newtheorem{Ass}{Assumption}

\newtheorem*{rem*}{Remark}

\allowdisplaybreaks
\colorlet{BLUE}{blue}

\begin{document}
\title[Anisotropic Lacunary Sums]{On a Problem of Kac concerning Anisotropic Lacunary Sums}
\author{Lorenz Fr\"uhwirth}
\author{Manuel   Hauke-Treuer }

\begin{abstract}
    Given a sequence of integers $(n_k)_{k \in \N}$ satisfying $\liminf_{k \rightarrow \infty}\frac{n_{k+1}}{n_k} \ge q > 1$, a scheme of real-valued weights $(c_{k,N})_{k \leq N}$, and a function   $f: \T = \R/\mathbb{Z} \rightarrow \R$, it is a classical problem to determine whether
    \[
    S_N(x):=\frac{\sum_{k=1}^N c_{k,N} f(n_k x)}{ \sqrt{\sum_{k=1}^N c_{k,N}^2}}
    \]
    converges to a normal distribution as $N \rightarrow \infty$. This question can be traced back to early works of Kac and Erd\H os and has since then been investigated. It is well known that the asymptotic behavior of $S_N$ is determined not only by the analytic properties of $f$, but also by the arithmetic properties of the sequence $ (n_k)_{k \in \N}$ and by the subtle interaction between the two. In this article, we identify conditions on $ (n_k)_{k \in \N}$, $f$ and $(c_{k,N})_{k \leq N}$ that can be regarded as minimal for a central limit theorem to hold for $S_N$. This addresses a question raised by Mark Kac [Ann. of Math., 1946].
\end{abstract}

\maketitle
\section{Introduction and Main Results}
Lacunary sums are classical objects at the crossroads of number theory, dynamical systems, and the asymptotic theory of weakly dependent random variables. In this setup, one considers an integer sequence $(n_k)_{k \in \N}$ satisfying the Hadamard-gap condition
\begin{equation}
    \label{eq:Hadamard_gap}
    \liminf_{k \rightarrow \infty} \frac{n_{k+1}}{n_k} \geq q > 1,
\end{equation}
and for a sufficiently regular function $f: \T = \R/\mathbb{Z} \rightarrow \R$ with mean zero, we   examine
\[
 \sum_{k =1}^N f(n_k x), \quad N \in \N,\, x \in \T.
\] 
For more than a century, mathematicians have been interested in the limiting behavior of these sums as $N \rightarrow \infty$. Early works trace back to Borel~\cite{Borel1909}, who studied sums $\sum_{k=1}^N f(2^k x)$ generated by certain indicator functions $f$, in order to understand the properties of normal numbers, and to Weyl~\cite{Weyl1916}, who was interested in the uniform distribution properties of sequences of the form $(n_k x)_{k \in \N}$ for $x \in \T$. Researchers soon recognized profound connections to other fields and started investigating lacunary sums as an independent topic. In 1924, Kolmogorov proved in his article~\cite{kolmogoroff1924} that
\[
\sum_{k =1}^\infty c_k \cos(2 \pi n_k x) 
\]
converges for almost all $x \in \T$ if one assumes the square summability of the coefficients $(c_k)_{k \in \N}$. Several years passed before researchers succeeded in establishing the first results on the scale of the central limit theorem (CLT). In his groundbreaking work~\cite{Kac1946Clt}, Kac proved in 1946 (see~\cite[Theorem 1]{Kac1946Clt}) that for any mean-zero function $f: \T \rightarrow \R$ whose Fourier coefficients decay sufficiently fast, we have
\[
\frac{1}{\sqrt{N}}\sum_{k =1}^N f(2^k x) \stackrel{w}{\longrightarrow} \mathcal{N}(0,\sigma^2)
\]
for some suitable   $\sigma = \sigma_f \geq 0$. Naively, one would expect the variance $\sigma^2$ to equal $\int_{\T} f(x)^2\, \mathrm{d}x$, which would be in accordance with true independent behavior; however, it turned out that
the resulting (asymptotic) variance is given by
\[
\sigma^2 = \int_{\T} f(x)^2 \,\mathrm{d}x + 2\sum_{k=1}^{\infty} \int_{\T} f(2^k x) f(x) \,\mathrm{d}x.
\]
This deviation from the i.i.d. random case arises from an intricate correlation structure, which, as later discovered, is inherent to certain dynamical systems. In the same article, as well as in his later work~\cite{K49}, Kac also studied superlacunary sequences where
\begin{equation}
\label{eq:large_gaps}
\lim_{k \rightarrow \infty} \frac{n_{k+1}}{n_k} = \infty.
\end{equation}
In this case, a CLT was established that perfectly aligns with the classical one for a sequence of independent, uniformly distributed random variables in $\T$, where the asymptotic variance is  indeed  given by $ \int_{\T} f(x)^2 \,\mathrm{d}x$. The question arose whether it is possible to obtain CLTs for lacunary sums $\sum_{k=1}^N f(n_k x)$ generated by an arbitrary integer sequence $(n_k)_{k \in \N}$ only satisfying \eqref{eq:Hadamard_gap} and $f$ coming from a reasonably large class of functions. It soon became clear that this goal was not attainable in full generality, even when the generating function $f:\T \rightarrow \R$ was restricted to being a trigonometric polynomial:   the  example of Erd\H os and Fortet (see~\cite[p.646]{K49}) demonstrates that choosing $f(x) = \cos(2 \pi x) + \cos(4 \pi x)$ and $n_k = 2^k - 1$ results in a lacunary sum $\sum_{k \leq N} f(n_k x)$ that satisfies a non-degenerate distributional limit theorem, with the asymptotic distribution being non-Gaussian. It took some time before the picture became clearer, as it was necessary to identify requirements on $(n_k)_{k \in \N}$ to prevent anomalies such as those that arise in the Erd\H os\,--\,Fortet example. In his influential article, Gaposhkin~\cite[Corollary~4]{G71} was able to show that, if for any fixed choice of  positive integers $a,b$ the number of solutions $k, \ell \in \mathbb{N}$  to the Diophantine equation 
\begin{equation}
\label{Eq:Dio_equation_second_order}
a n_k - bn_\ell =c
\end{equation} 
 is uniformly bounded in $c \in \mathbb{N}$, then $\sum_{k \leq N} f(n_k x)$  satisfies a CLT  when normalized by its $L^2$-norm, where $f$ can be taken from a reasonably large class of functions, provided that $\int_{\T} \left(\sum_{k \leq N}f(n_kx)\right)^2 \,\mathrm{d}x \geq \delta N$ for some $\delta > 0$. About 40 years later, Aistleitner and Berkes reduced Gaposhkin's requirements to a minimal level. 
In their  article~\cite{AB10}, they demonstrated that under the conditions  
\begin{equation} 
\label{eq:Dio_condition_isotropic}  
\sup_{c \in \N} \# \{ k, \ell \leq N : j n_k - j' n_\ell = c  \} = o_d(N), \quad 1 \leq j,j' \leq d \in \N,  
\end{equation} 
and
\begin{equation}
\label{eq:lower_bound_variance}
\Big \lVert \sum_{k \leq N} f(n_k x) \Big \rVert_2 \geq \delta \sqrt{N},
\end{equation}
for some $\delta> 0$, one has
\begin{equation*}
 \frac{\sum_{k \leq N} f(n_k x)}{ \lVert \sum_{k \leq N} f(n_k x) \rVert_2} \stackrel{w}{\longrightarrow} \mathcal{N}(0,1),  
\end{equation*}
whenever $f: \T \rightarrow \R$ is an arbitrary mean-zero function of bounded variation (see \cite[Theorem 1.1]{AB10}). Alternatively,   assume  that for all $1 \le j,j' \le d \in \N$, we have
\begin{equation}
\label{eq:inho_homo_solutions}
\sup_{c \in \N} \# \{ 1 \leq k , \ell \leq N : j n_k - j' n_\ell = c  \}  + \# \{ 1 \le k< \ell \leq N: j n_k - j' n_\ell =0 \} =o_d(N),
\end{equation}
 This  ensures a CLT that is in accordance with the classical one for i.i.d. random variables, i.e., one has for any $f: \T \rightarrow \R$ of bounded variation, with mean-zero and $\lVert f\rVert_2 > 0$ that  (see \cite[Theorem 1.2]{AB10})
  \[
\frac{\sum_{k=1}^N f(n_kx)}{\lVert f \rVert_2 \sqrt{N}} \stackrel{w}{\longrightarrow} \mathcal{N}(0,1).
\] 
In~\cite[Theorem 1.3]{AB10}, the authors also show that the condition~\eqref{eq:Dio_condition_isotropic} is minimal in order to guarantee the existence of a CLT. 

We remark that several other laws known for the classical i.i.d. random case have been examined for lacunary sums: Diophantine conditions such as the one in~\eqref{eq:Dio_condition_isotropic} were investigated on the scale of the law of the iterated logarithm, see, e.g., ~\cite{aistleitner_LiL_2010,AFP_LiL_2024}. Furthermore, in recent years the study of lacunary sums has been extended even further to the scale of moderate and large deviations revealing an unexpected asymptotic behavior not visible on the scales of the CLT and the law of the iterated logarithm (see~\cite{AFHM25,AGKPR_ldp_20,FJP22,PS_moderate_23}). For a more comprehensive discussion on the foundations of lacunary sums and their connections to various topics in modern mathematics, we refer the interested reader to the  recent survey article~\cite{ABT_survey_2023}.\\

Let us   now return our attention  to the more classical question of CLTs which is the purpose of this article. In the spirit of Kolmogoroff, Kac was even able to derive CLTs for anisotropic (i.e.,  weighted) lacunary sums in his influential 1946 paper~\cite{Kac1946Clt}; these take the form
  \begin{equation}\label{def_SN}
S_N := \sum_{k \leq N} c_{k,N} f(n_k x),
\end{equation} 
where the triangular scheme $(c_{k,N})_{k \leq N}$ satisfies a Lindeberg\,--\,Feller-type condition, i.e.,

\begin{equation}
\label{eq:lindeberg_weights}
\sup_{N \in \mathbb{N}}\sup_{k \leq N} |c_{k,N}| < \infty \quad \text{and} \quad  \sum_{k \leq N} c_{k,N}^2 \stackrel{N \rightarrow \infty}{\longrightarrow} \infty.
\end{equation}
Indeed, writing $B_N^2=\sum_{k\leq N}c_{k,N}^2$, condition~\eqref{eq:lindeberg_weights} implies $\max_{k\leq N}|c_{k,N}|/B_N\to0$, corresponding to the Feller condition, while the Lindeberg condition follows from $f\in L^2(\T)$ and the uniform boundedness of the coefficients.
 
Although Kac was able to deal with rather general functions $f$, CLTs were only obtained for certain superlacunary sequences $(n_k)_{k \in \N}$. In the same article \cite{Kac1946Clt} -- following a communication with Paul Erd\H os who pointed out a way to prove a CLT for $\sum_{k \leq N } c_k f(2^k x)$-- Kac writes

\begin{center}
``[...] no way has as yet been found to
answer the questions discussed in this paper if the sequence 
$2^k$ is replaced by a
general gap sequence $n_{k+1}/n_k > q > 1$.''
\end{center}
This means that Kac asked for CLTs for 
\begin{equation}
\label{eq:Kac_general_clt}
\frac{\sum_{k \leq N } c_k f(n_k x)}{ \lVert \sum_{k \leq N } c_k f(n_k x) \rVert_2},
\end{equation}
where $(n_k)_{k \in \N}$ is an \textit{arbitrary} lacunary sequence only satisfying Hadamard's condition~\eqref{eq:Hadamard_gap}, the weights $(c_k)_{k \in \mathbb{N}}$ satisfy the Lindeberg\,--\,Feller condition\eqref{eq:lindeberg_weights}, and $f$ is taken from a reasonably large class of functions (that has a Fourier decay property described later).
For general $f$ and without assuming restrictive assumptions on $(n_k)_{k \in \N}$, progress on that problem appears to have been limited, unless one assumes a particular structure such as $n_k = q^k, q \in \mathbb{N}_{\geq 2}$. The $q$-adic lacunary sequences $n_k = q^k$, where $q \in \N_{\geq 2}$, play a special role since one can relate lacunary sums of the form $\sum_{k \in \N} f(q^k x)$ to the theory of dynamical systems via interpreting $q^k x$ as the $k$-fold composition of the $q$-adic shift-operator $\mathcal{T}_q: \T \rightarrow \T$ defined by $x \mapsto q x$. In recent years, it was shown in~\cite{conze2007limit} that $\mathcal{T}_q \circ \ldots \circ \mathcal{T}_q$, where the $\mathcal{T}^j$ belong to suitable families of piecewise expanding maps of the interval, can be replaced by a composition of potentially different operators $\mathcal{T}^1 \circ \mathcal{T}^2 \circ \ldots$, and additionally the function $f$ can be replaced by a sequence of functions $f_1, f_2, \ldots$, under some reasonably mild additional conditions:   there, a CLT is established for random variables of the form
\[
\frac{\sum_{k \leq N} f_k( \mathcal{T}^k \circ \ldots \circ \mathcal{T}^1(x))}{ \big \lVert \sum_{k \leq N} f_k( \mathcal{T}^k \circ \ldots \circ \mathcal{T}^1(x))  \big \rVert_2}, \qquad x \in \T, \quad N \in \N.
\]
In particular, the findings in~\cite{conze2007limit} imply a CLT for sequences of the form 
\[
\frac{\sum_{k \leq N } c_k f(q^k x)}{ \big \lVert \sum_{k \leq N } c_k f(q^k x) \big \rVert_2},
\]
as was claimed for $q=2$ in Kac's  article \cite{Kac1946Clt}.
The existing isolated results primarily address the special cases where $f(x) = \cos(2 \pi x)$ (see~\cite{salem_zyg_1947}) or $f(x) = \sign(\cos(2 \pi x))$. In the latter case, for even $q$ the random variables $\sign(\cos(2\pi q^k x))$ are i.i.d. Rademacher variables, so that the corresponding weighted CLT follows directly from the classical Lindeberg--Feller theorem. More general CLTs for lacunary Walsh series were obtained, e.g., by Morgenthaler in~\cite{morgenthaler1957}. In the article~\cite{G71} Gaposhkin established CLTs for $\sum_{k \leq N} c_{k,N} f( n_k x)$, where $(c_{k,N})_{k \leq N \in \N}$ is a triangular scheme of weights satisfying a condition analogous to~\eqref{eq:lindeberg_weights} and $f$ can be taken from a large class of functions. However, the sequence $(n_k)_{k \in \N}$ needs to satisfy very restrictive Diophantine conditions. 
\par{}
The objective of the present article is to
establish CLTs for a broad class of anisotropic lacunary sums as in~\eqref{eq:Kac_general_clt} under minimal Diophantine conditions. Our main results (Theorems~\ref{thm:general_clt} and~\ref{thm:iidclt}) represent significant progress on the original question posed by Kac in~\cite{Kac1946Clt}, aligning the theory of anisotropic CLTs with the isotropic case established by Aistleitner and Berkes~\cite{AB10}.

\subsection{Main Results}
Before presenting our main results, we introduce the assumptions required to establish Theorem~\ref{thm:general_clt} and Theorem~\ref{thm:iidclt}.

\begin{Ass}
\label{Ass_f}
Let $f : \T \rightarrow \R$ be a function in $\mathbb{L}_2(\T)$ with
\begin{equation*}
\int_{\T} f(x) \,\mathrm{d}x = 0, \qquad \int_{\T} f^2(x) \,\mathrm{d}x = \sigma^2 \in (0, \infty).
\end{equation*}
Writing $f(x) = \sum_{j=1}^{\infty} \left[ a_j \cos(2 \pi j x) + b_j \sin(2 \pi j x) \right]$  for some $(a_j)_{j\in \mathbb{N}},(b_j)_{j \in \mathbb{N}} \subseteq \mathbb{R}$, we shall assume that there are absolute constants $M > 0$ and $\rho > \frac{1}{2}$ such that
\[
| a_j| + |b_j| \leq \frac{M}{j^{\rho }}, \quad j \in \N.
\]
\end{Ass}

\begin{rem}
The assumption of $f: \T \rightarrow \R$ having a finite second moment   ensures that the normalizations used in the CLTs below are well defined. The requirement for its Fourier coefficients to decay at a certain rate is the same as in Kac's  original formulation~\cite{Kac1946Clt}. Since  every function of bounded variation satisfies the above, one could see this as a slight improvement over~\cite{AB10}, although their method could be adapted to cover this setting as well. Moreover, we remark that any Hölder-continuous function with exponent $\rho > \frac{1}{2}$ has Fourier-coefficients tending to $0$ at speed $O (j^{-\rho} )$.
\end{rem}
Our methods are not limited to a fixed sequence of weights $(c_k)_{k \in \mathbb{N}}$; they also allow for more general triangular schemes $(c_{k,N})_{k \le N \in \mathbb{N}}$, including weights that may take negative values within the range $[-1,1]$. The following assumption represents a natural generalization   of  the Lindeberg\,--\,Feller type condition from~\cite{Kac1946Clt} to triangular arrays of signed weights.

\begin{Ass}
\label{Ass:Condition_weights}
The array of  weights $(c_{k,N})_{k \leq N \in \N}$ satisfies 
\begin{equation}
\sup_{N \in \N} \sup_{k \leq N} |c_{k,N}| \leq 1, \quad h(N) = \sum_{k \leq N} c_{k,N}^2 \longrightarrow \infty \quad \text{as} \quad N \rightarrow \infty.
\end{equation}
\end{Ass}

The results in the works~\cite{aistleitner_LiL_2010,AB10,AFP_LiL_2024} suggest that limiting the number of solutions to certain Diophantine equations is necessary to obtain a CLT for suitably normalized lacunary sums. The following assumption not only imposes a limit on the number of Diophantine solutions to equations  
  \[
j n_k - j' n_\ell = c,  \quad j,j',c \in \mathbb{Z},
\]   
 with $|j|,|j'|$ bounded,  but also captures the subtle interplay between the weights $(c_{k,N})_{k \le N \in \N}$, and the given lacunary sequence $(n_k)_{k \in \mathbb{N}}$. Heuristically speaking, if there are not too many \textit{inhomogeneous} (i.e.,  \ $c \neq 0$) solutions to the equation above (and some additional lower bound on the variance exists), then the CLT holds with a non-degenerate variance. We remark that this bound on the inhomogeneous solutions is not satisfied in the Erd\H os\,--\,Fortet example and can be seen as the technical reason \textit{why} the CLT fails in that instance. If additionally the number of solutions to the \textit{homogeneous} equation (i.e., \ where $c=0$) is ``not too large'', then we expect the CLT to hold with asymptotic variance $\int_0^1 f^2(x)\,\mathrm{d}x$. We remark that this phenomenon was accurately described in~\cite{aistleitner2013lil} on the scale of the law of the iterated logarithm, and we refer the reader to~\cite{aistleitner2013lil} for details in this direction. The bounds on the number of solutions in the homogeneous and inhomogeneous cases assumed in the present article are as follows:

\begin{Ass}
\label{Ass:dioph_cond_dream}
For any fixed $d \in \N$, we assume that the array of weights \\$(c_{k,N})_{k \leq N \in \N}$ and the sequence $(n_k)_{k \in \N}$ satisfy 
  \begin{equation*}
L(N,d) = \sup_{c \in \N} \sum_{1 \leq j,j' \leq d} \sum_{k,\ell \leq N}  |c_{k,N}| |c_{\ell,N}| \mathds{1}_{[jn_k-j'n_{\ell} = c]} = o(h(N)),
\end{equation*} 
and 
\begin{equation}
\label{non_deg_var}
\Big \lVert \sum_{k \leq N}c_{k,N}f(n_kx) \Big \rVert_2 \geq \delta \sqrt{h(N)}\end{equation}
for some absolute constant $\delta > 0$.
\end{Ass}

\begin{Ass}
\label{Ass:dioph_cond_homog}
For any fixed $d \in \N$, we assume
  \[
L^*(N,d) = L(N,d)+ \sum_{1 \leq j,j' \leq d} \sum_{ \substack{1 \leq k,\ell \leq N \\ k \neq \ell }} |c_{k,N}| |c_{\ell,N}|  \mathds{1}_{[jn_k-j'n_{\ell} = 0]} = o(h(N)).
\] 
\end{Ass}

\begin{rem}
While Assumption \ref{Ass_f} is a slight generalization   of the assumtion  that $f: \T \rightarrow \R$ is of bounded variation, Assumptions \ref{Ass:dioph_cond_dream} and \ref{Ass:dioph_cond_homog} are   a somewhat  natural generalization of the conditions imposed in~\cite[Theorem 1.1]{AB10} and~\cite[Theorem 1.2]{AB10} to  weighted lacunary sums: indeed, in the isotropic case, Assumptions \ref{Ass:dioph_cond_dream} and \ref{Ass:dioph_cond_homog} coincide with the respective conditions assumed  in~\cite[Theorem 1.1]{AB10} and~\cite[Theorem 1.2]{AB10}.
  Hence, as  stated in Corollary \ref{Cor:Clt_CA} below, the latter two theorems directly follow from our main results.
\end{rem}
With these assumptions at hand, we present the main results of our article.

\begin{thm}
\label{thm:general_clt}
Let $(n_k)_{k \in \N}$ be a lacunary sequence, let $(c_{k,N})_{k\leq N \in \N}$ be a triangular array of reals, and let $f: \T \rightarrow \R$ be such that Assumptions~\ref{Ass_f},~\ref{Ass:Condition_weights} and~\ref{Ass:dioph_cond_dream} are satisfied. Then we have
\[
\frac{\sum_{k \leq N}c_{k,N}f(n_kx)}{\big \lVert \sum_{k \leq N}c_{k,N}f(n_kx) \big \rVert_2} \stackrel{w}{\longrightarrow} \mathcal{N}(0,1).
\]
\end{thm}

\begin{thm}
\label{thm:iidclt}
Let $(n_k)_{k \in \N}$ be a lacunary sequence, let $(c_{k,N})_{k\leq N \in \N }$ be a triangular array of reals, and let $f: \T \rightarrow \R$ be such that Assumptions~\ref{Ass_f},~\ref{Ass:Condition_weights} and~\ref{Ass:dioph_cond_homog} are satisfied. Then we have
\[
\frac{\sum_{k \leq N}c_{k,N}f(n_kx)}{\lVert f \rVert_2 \sqrt{h(N)}} \stackrel{w}{\longrightarrow} \mathcal{N}(0,1).
\]
\end{thm}

\begin{rem}
\label{rem:dio_condition}
\begin{enumerate}
\item 
Even without Assumptions~\ref{Ass:dioph_cond_dream} or~\ref{Ass:dioph_cond_homog}, it follows from the Fourier estimate in Section~\ref{sec_final_CLT} that  (recall \eqref{def_SN})
under the Assumptions~\ref{Ass_f} and ~\ref{Ass:Condition_weights}, 
 $\lVert S_N \rVert_2 = O ( \sqrt{h(N)} )$. On the other hand, under Assumption~\ref{Ass:dioph_cond_dream}, we have 
\[
 0 < \frac{\lVert S_N \rVert_2}{\sqrt{h(N)} \lVert f \rVert_2} = O( 1),
\]
where the positivity is guaranteed by~\eqref{non_deg_var}. We emphasize that analogous assumptions to~\eqref{non_deg_var} were also employed to obtain CLTs in~\cite{AB10, G66,G71,Kac1946Clt}.

\item 
 Having  too many solutions to either the homogeneous Diophantine equation \( j n_k - j' n_\ell = 0 \) or to inhomogeneous equations of the form \( j n_k - j' n_\ell = c \), for some \( c \in \mathbb{N} \), potentially leads to the failure of a central limit theorem analogous to that for sums of i.i.d. random variables.
\par 
For example, for lacunary sums of the form \( \sum_{k=1}^N f(q^k x) \) for some \( q \in \mathbb{N}_{\ge 2} \), the presence of many Diophantine solutions for \( c = 0 \) implies that
  \[
\frac{\lVert S_N \rVert_2}{\sqrt{N}} = \sigma_f(q) + o(1),
\] 
where, in general,   \(\sigma_f(q) \neq \lVert f \rVert_2 \). There are also examples of lacunary sequences \( (n_k)_{k \in \mathbb{N}} \) that yield too many solutions to \( j n_k - j' n_\ell = c \) for some \( c \in \mathbb{N} \), which can likewise lead to the failure of a CLT, as shown in~\cite[Theorem~1.3]{AB10}. 
\end{enumerate}
\end{rem}

As a first corollary from Theorem~\ref{thm:general_clt} and Theorem~\ref{thm:iidclt}, we can derive the main results in~\cite{AB10}.
\begin{cor}
\label{Cor:Clt_CA}
Let $f: \T \rightarrow \R$ be a mean-zero function of bounded variation  with $\lVert f\rVert_2 > 0$.
\begin{enumerate}
    \item If a lacunary sequence $(n_k)_{k \in \N}$ satisfies~\eqref{eq:Dio_condition_isotropic} and~\eqref{eq:lower_bound_variance}, then
    \[
 \frac{\sum_{k \leq N} f(n_k x)}{ \lVert \sum_{k \leq N} f(n_k x) \rVert_2} \stackrel{w}{\longrightarrow} \mathcal{N}(0,1).
    \]
    \item If a lacunary sequence $(n_k)_{k \in \N}$ satisfies~\eqref{eq:inho_homo_solutions}, then
    \[
    \frac{\sum_{k \leq N} f(n_k x)}{ \lVert f \rVert_2 \sqrt{N}} \stackrel{w}{\longrightarrow} \mathcal{N}(0,1).
    \]
\end{enumerate}
\end{cor}
\begin{proof}
In the isotropic case where $c_{k,N}=1$ for all $k \leq N \in \N$, we see that Assumption~\ref{Ass:dioph_cond_dream} coincides with~\eqref{eq:Dio_condition_isotropic} and~\eqref{eq:lower_bound_variance}. Thus, (1) is a consequence of Theorem~\ref{thm:general_clt}.   Similarly, Assumption~\ref{Ass:dioph_cond_homog} becomes the condition in~\eqref{eq:inho_homo_solutions} and consequently (2) follows from Theorem~\ref{thm:iidclt}.
\end{proof}

An instance of a CLT for anisotropic lacunary sums can be found in~\cite{K49}, one of the very few such results in the literature. It may also be deduced directly from our main theorems:

\begin{cor}
\label{cor_superlac}
  Let  $(n_k)_{k \in \N}$ be a superlacunary sequence, i.e.,
 \begin{equation*}
\lim_{k \rightarrow \infty} \frac{n_{k+1}}{n_k}= \infty 
\end{equation*}
and let $f: \T \rightarrow \R$ be a mean-zero function of bounded variation  with $\lVert f \rVert_2 > 0$. Then, for any sequence of weights $(c_k)_{k \in \N}$
that   satisfies Assumption \ref{Ass:Condition_weights}
(or, more formally: an array of weights $(c_{k,N})_{k \leq N}$ with $c_{k,N} = c_k$ independent of $N$),
 it holds that
\[
\frac{\sum_{k=1}^N c_k f(n_kx)}{ \lVert f \rVert_2\sqrt{h(N)}} \stackrel{N \rightarrow \infty}{\longrightarrow} \mathcal{N}(0,1).
\]
\end{cor}
\begin{proof}
Theorem~\ref{thm:iidclt} directly implies the above result, since it can be easily shown that any superlacunary sequence $(n_k)_{k \in \N}$ satisfies Assumption~\ref{Ass:dioph_cond_homog}.
\end{proof}

\subsection{Notation}
In this article we deal with random variables $S_N$, $N \in \N$ which are defined on the probability space $(\T, \mathcal{B}(\T), \mathbb{P} )$, where $\T$ is the $1$-dimensional torus, $\mathcal{B}$ denotes the Borel-$\sigma$-algebra and $\mathbb{P}$ denotes the normalized Haar measure on $\T$. We write $\mathbb{E}[S_N]$ for the expected value
\[
\mathbb{E}[S_N]= \int_{\T } S_N(x) \, \mathrm{d} \mathbb{P}(x) = \int_{\T } S_N(x) \, \mathrm{d} x.
\]
We denote convergence in distribution by $\stackrel{w}{\longrightarrow}$ and convergence in probability by $\stackrel{\mathbb{P}}{\longrightarrow}$. 
When dealing with a deterministic sequence, we will just write $\longrightarrow$ to indicate convergence. For a random variable $S_N$ on $(\T, \mathcal{B}(\T), \mathbb{P})$ we write
\[
\lVert S_N \rVert_2 = \left( \int_{\T} S_N^2(x) \,\mathrm{d}x \right)^{1/2},
\]
and   $\lVert S_N \rVert_\infty = \operatorname{ess}\sup_{x \in \T} | S_N(x) |$. The collection of random variables $X$ on $(\T, \mathcal{B}(\T), \lambda)$ with $ \lVert X \rVert_2 < \infty$ is denoted by $\mathbb{L}_2(\T)$.   For  $\mathcal{F} \subseteq \mathcal{B}(\T)$  a sub-$\sigma$-algebra,  we denote the conditional expectation of a random variable $X \in \mathbb{L}_2(\T)$ with respect to $\mathcal{F}$ by $\mathbb{E}\left[ X | \mathcal{F} \right]$.  
We denote the standard normal distribution function as $\Phi$, i.e., we have
\[
\Phi(t) = \frac{1}{\sqrt{2 \pi }} \int_{- \infty}^t e^{-\frac{x^2}{2}} \,\mathrm{d}x. 
\]
For real-valued sequences $(a_N)_{N \in \N}$ and $(b_N)_{N \in \N}$ we write 
\[
a_N = O(b_N) \quad \text{or} \quad a_N \ll b_N,
\]
if there exists a constant $c> 0$ such that $\lvert a_N \rvert \leq c \lvert b_N \rvert $ for all but finitely many $N \in \N$. If 
$\lim_{N \rightarrow \infty} \frac{a_N}{b_N} = 0$, we write $a_N = o(b_N)$, and if $a_N = O(b_N)$ and $b_N = O(a_N)$, we write $a_N \asymp b_N$. We will use the notations $a_N = \Omega(b_N)$ and $a_N = \omega(b_N)$ if and only if $b_N = O(a_N)$ and $b_N = o(a_N)$, respectively. For a finite set $A \subseteq \mathbb{Z}$, we denote the number of elements in $A$ by $\# A$.

\subsection{Acknowledgments}
 LF is partially supported by the DFG project \textit{Limit theorems for the volume of random projections of $\ell_p$-balls} (project number 516672205).   MHT  was
funded in whole or in part by the Austrian Science Fund (FWF)
10.55776/ESP5134624. A significant part of this work was carried out during visits   by  LF at NTNU Trondheim   and by MHT  at the University of Passau. We thank both universities for their support and hospitality. We would also like to thank Christoph Aistleitner and Joscha Prochno for various helpful comments on an earlier version of this manuscript. We  thank the anonymous   referees  for valuable comments   that improved the readability of this manuscript, and  for noting that our results remain valid without any restriction on the sign of   $c_{k,N}$.

\subsection{AI declaration}
We generated the figures with the help of AI (Chat GPT-6 Astra). All other parts of this manuscript are solely the work of the authors.

\section{Proofs}

 \subsection*{Notational simplifications}
From now on, we will always deal with an array of weights $(c_{k,N})_{k \leq N, N \in \N}$ satisfying Assumption~\ref{Ass:Condition_weights}. 
As all arguments are obtained for fixed $N$, we suppress for better readability the index $N$ in $c_{k,N}$, and write $c_k$ instead. Further, we replace the Hadamard-gap condition $\liminf_{k \to \infty} \frac{n_{k+1}}{n_k} \geq q > 1$ with
$\frac{n_{k+1}}{n_k} \geq q > 1$ for all $k \geq 1$; this is of course, strictly speaking not true --- in fact, we need to remove the first finitely many, until the we reach an index $K_0$ such that for all $k \geq K_0$, we have $\frac{n_{k+1}}{n_k} \geq \frac{q+1}{2} > 1$. By renaming $q$ and arguing that a the first finitely many indices do not influence the limiting behaviour, we justify our choice of $\frac{n_{k+1}}{n_k} \geq q > 1$ for all $k \geq 1$.
Further, we will omit the dependence of implied constants arising on fixed polynomials $p$, their degrees, or the variable $q$ from the Hadamard-gap condition.

 \subsection{Heuristic ideas and novelties of the proof}
\label{sec:heuristic}

The main challenge is to obtain a CLT for lacunary sums $S_N = S_N(f)$ when $f$ is a trigonometric polynomial of fixed degree $d$. The transition to arbitrary functions with Fourier decay as in Assumption~\ref{Ass_f} then follows from an approximation argument.
The overall strategy to achieve the CLT for a trigonometric polynomial is based on the approach in~\cite{AB10} where a filtered probability space and a martingale that approximates $S_N$ reasonably well are constructed. Since CLTs for such adapted processes are well understood (see Proposition~\ref{prop:martingale_clt}), this allows us to obtain a CLT for $S_N$ being generated by a trigonometric polynomial of finite degree. 
\par{}
On a more technical level, one partitions the set of integers $[1,N]$ into two types of blocks. On the one hand, there are ``major'' blocks $\Delta_i$ that should contain most of the \textit{mass} in terms of the weights   $c_{k,N}^2$. 
On the other hand, between each $\Delta_i$ and $ \Delta_{i+1}$, there are ``buffer blocks'' $\Delta_i'$ that are reasonably \textit{long} ensuring that   $\sum_{k \in \Delta_i}c_{k,N}f(n_kx)$ and $\sum_{k \in \Delta_j}c_{k,N}f(n_kx)$  for $i \neq j$ are  ``close to being independent''. This allows us to establish a CLT on the blocks $\Delta_i$, which implies a CLT on the entire lacunary sum, provided the contribution from the $\Delta_i'$ is negligible. \\
Defining the blocks $\Delta_i, \Delta_i'$ in a suitable way is one of the key difficulties when moving from the isotropic case of~\cite{AB10} to the anisotropic case considered here.
A crucial part of the arguments used in~\cite{AB10} relies on the fact that the ``mass'' of each block $\Delta_i$ is proportional to its number of elements, i.e.,
\begin{equation}\label{compare_mass}
\sum_{k \in \Delta_i}1 = \sum_{k \in \Delta_i}c_k^2,
\end{equation} 
which is implicitly used at many places   in  the proof, and most importantly in the above-mentioned buffer block construction.
In the general anisotropic case, a direct mirroring of this approach is doomed to fail:   a  buffer block $\Delta_i'$ separating $\Delta_i$ and $\Delta_{i+1}$ needs to be at least of a specific length (relative to the masses respectively cardinalities of $\Delta_i$ and $\Delta_{i+1}$).
  If the precise position  of $\Delta_i'$   is not adjusted with respect to the weights $c_k, k \in \Delta_i'$, this can  lead to $\Delta_i'$ consisting of far too much mass.
\par{}
We address this issue by an intricate inductive procedure:   we  split our weights $(c_k)_{k \le N}$ into groups of ``large'' and ``small'' ones in a way such that these groups are well separated. In the group consisting of the ``small'' $c_k$, adding buffer blocks of a suitable length is unproblematic since the weights in these buffer blocks are controllable. Therefore, the blocks of "small" $c_k$ can be further split into blocks $\Delta_i$ that have a reasonable mass. For the group of ``larger'' $c_k$ we repeat this procedure, again splitting into ``small'' and ``large'' (now at a different scale).
In   this  way, we partition the   retained indices into groups
$A_j$ on which the magnitudes of the weights satisfy suitable
upper bounds. These bounds allow us to control the mass
discarded when inserting buffer blocks within each group, allowing to circumvent an assumption such as  \eqref{compare_mass}.
Furthermore, we create additional buffer blocks between $A_j$ and $A_i$ for $i \neq j$ by a subtle argument comparing the potential cardinalities of $A_j$ and $A_i$. In the end, this will give us ``mass'' blocks $\Delta_i$ that are well separated from each other by suitable buffer blocks of negligible contribution. This construction together with a careful error analysis is somewhat involved and requires several pages. The entire procedure is carried out in Subsection \ref{subsec:blocks_filtration}.

\subsection{Prerequisites}

We establish the CLT in Proposition~\ref{prop:clt_cosine_big_ck} using methods from (discrete) martingale theory. This will need some preparation which is done in the upcoming subsections. Here, we will just state our main auxiliary result below.

\begin{prop}(\cite[Theorem 1, with $\delta=1$]{heyde_martingale_clt})
  \label{prop:martingale_clt}
  Let $\left( X_M \right)_{M \in \N } $ with $ X_M = \sum_{i=1}^M Y_i$ be a martingale adapted to a filtration $\left(\mathcal{F}_M \right)_{M \in \N }$ such that all $Y_i$ have finite fourth moments. Let $V_M = \sum_{i=1}^M \mathbb{E} \left[ Y_i^2 | \mathcal{F}_{i-1} \right]$ and define the sequence $(s_M)_{M \in \N}$ via $s_M = \sum_{i=1}^M \mathbb{E} \left[ Y_i^2 \right]$. Recalling that $\Phi$ denotes the standard normal distribution function, we have
  \[
  \sup_{t \in \R} \left| \mathbb{P} \left[ \frac{X_M}{\sqrt{s_M}} \leq t \right] - \Phi(t) \right| \ll \left(  \frac{\sum_{i=1}^M \mathbb{E} \left[ Y_i^4 \right] + \mathbb{E} \left[ (V_M -s_M)^2 \right]}{s_M^2}\right)^{1/5},
  \]
  where the implicit constant is absolute. 
\end{prop}

\begin{lem}(\cite[Lemma 2.2]{AB10})
\label{riemann_lebesgue}
Let $f: \T \rightarrow \R$ satisfy Assumption~\ref{Ass_f}. Then for any $a,b \in \T$ and for any $\lambda > 0$ we have
\[
\left| \int_{[a,b]}f(\lambda x) \,\mathrm{d}x \right| \leq \frac{1}{\lambda}\int_{\T}\lvert f(x) \rvert \,\mathrm{d}x \leq \frac{1}{\lambda}\lVert f \rVert_{\infty}.
\]
\end{lem}

\subsection{Construction of the filtration and an adapted process}
\label{subsec:blocks_filtration}

In the following, we will put some effort into the construction of sets $\Delta_1, \Delta_2, \ldots, \Delta_M$ and $\Delta_1' , \Delta_2', \ldots$ such that
\[
\Delta_1 < \Delta_1' < \Delta_2 < \ldots
\]
holds element-wise, where the $\Delta_i'$ have a suitable length. The following lemma serves as a first step in separating "large" weights from "smaller" ones without creating an excessive error. 
 We start with the construction of a filtration and an adapted process, with the notation defined below being frequently used in the subsequent auxiliary results.

\begin{lem}
    \label{lem:Independent_sets}
    Let   $H \subseteq [1,N] = \{1,\ldots,N\}, m = \#H$ be sufficiently large, and let $(c_k)_{k\in H}$ be a scheme of weights such that $\sum_{k \in H} c_k^2 \leq h$ for some $h > 0$. Then we can find a partition $H = A \cup B \cup C$ with the following properties: for $m$ sufficiently large, there exists   $\xi \in \left[\frac{1}{4}, \frac{1}{3} \right)$  such that
    \begin{itemize}
        \item[(i)] All $k \in A$ satisfy $\lvert c_k \rvert  \le m^{- \xi - \log (m)^{-1/2}}$.
        \item[(ii)] All $k \in B $ satisfy $ \lvert c_k \rvert \ge m^{-\xi }$.
        \item[(iii)] We have   $ |k- \ell| \ge \exp \left(  \frac{\sqrt{\log(m)}}{2} \right)$ for all $k \in A, \ell \in B$.
        \item[(iv)] We have   $\sum_{k \in C} c_k^2 \ll \frac{h}{\sqrt{\log m}}$.
    \end{itemize}
\end{lem}

\begin{proof}
  We fix  some $\delta \in \left(0, \frac{1}{15} \right)$ and define
  \[
C_m = \left \{k \in H: m^{-1/4-\delta} < \lvert c_k \rvert \leq m^{-1/4} \right \}
\] 
Further, let $L = \lfloor \delta \sqrt{\log m}\rfloor$ and partition $C_m$ into 
     $C_m= \bigcup_{\ell = 0}^{L-1} C_{m,\ell} $, where
     \[
     C_{m,\ell} = \left \{k \in C_m: m^{-1/4-\delta(\ell+1)/L} < \lvert c_k \rvert \leq m^{-1/4-\delta \ell/L} \right \}, \quad \ell =0, \ldots, L-1.
     \]
     Since trivially $\sum_{k \in C_m}c_k^2 \leq h$, by the pigeonhole principle there is some $\ell_{0} \in \{0,\ldots,L-1\}$ such that
     \[
     \sum_{k \in C_{m,\ell_0}}c_k^2 \leq \frac{h}{L}.
     \]
Now let   $\xi = \frac{1}{4} + \delta \frac{\ell_0}{L} \in \left[\frac{1}{4}, \frac{1}{3} \right)$ and define
\[
\tilde{A} = \left \{k \in H: \lvert c_k \rvert \leq m^{-\xi - \frac{\delta}{L}} \right \}, \quad B  = \left \{k \in H: \lvert c_k \rvert > m^{-\xi} \right \}.
\] 
We now set $J = m^{\delta/(2L)}$ and let
\[
A' = \left \{k \in \tilde{A}: \exists \ell \in B: |k - \ell| \leq J \right \}, \quad A = \tilde{A} \setminus A'.
\]
We get
\[
\sum_{k \in A'}c_k^2 \leq 2J m^{-2\xi - \frac{2\delta}{L}} \# B
\]
and 
\[
\# B  \leq m^{2\xi} \sum_{k \in B} c_k^2 \leq m^{2\xi} h.
\]
Thus, we have
  \[
\sum_{k \in A'}c_k^2 \le 2 h J m^{-\frac{2\delta}{L}}
\leq 2h m^{-\frac{\delta}{L}} \leq 2 h \exp \left(- \sqrt{\log m} \right).
\] 
The proof is now concluded by setting $C=C_{m,\ell_0} \cup A'$.  \vspace{2cm}
 \end{proof}

 \begin{figure}[H]
    \centering
    \input{tikz1.tex}
    \label{fig:construction}
    \caption{A schematic illustration of the block construction, part 1. This is Lemmas \ref{lem:Independent_sets} and Lemma \ref{lem:Gamma}.}
\end{figure}

\clearpage

\begin{figure}[p]
    \centering
    \input{tikz2.tex}
     \caption{A schematic illustration of the block construction, part 2. This is the final decomposition from Lemma \ref{lem:Construction_blocks}.}
\end{figure}

 \begin{lem}
\label{lem:Gamma}
  For all $N$ sufficiently large, there exists a $T \in \mathbb{N} \cup \{0\}$, a decreasing sequence   $(N_j)_{j =0}^T \subseteq \mathbb{N} \cup \{0\}$, $ (\xi_j)_{j = 0}^{T-1} \subseteq [\tfrac{1}{4},\tfrac{1}{3})$, and $\Gamma \subseteq [1,N]$ such that $\Gamma = \dot\bigcup_{j=0}^T A_j $ with the following properties:
    \begin{itemize}
        \item[(i)] $\sum_{k \in [1,N] \setminus \Gamma}c_k^2 = o(h(N))$.
        \item[(ii)] $\#A_j \leq N_j$ for all $j \le T$.
        \item[(iii)] $\forall k \in A_j: \lvert c_k \rvert \leq N_j^{-\xi_j-\log(N_j)^{-1/2}}$ for $j=0, \ldots, T-1$.
        \item[(iv)] For all $0 \leq i < j \leq T$,  $\min\{|k-\ell|: k \in A_i, \ell \in A_{j}\} \geq \exp\left( \frac{\sqrt{\log N_i}}{2}\right)$.
        \item[(v)] We have $N_T < h(N)^{100}$ and $N_j \ge h(N)^{100}$ for all $0 \leq j < T$.
    \end{itemize}
\end{lem}

\begin{proof}
  Since $N$ is sufficiently large, we may assume $h(N) \geq 2$.  First, let us assume $ N \ge h(N)^{100}$. We work inductively as follows: We apply Lemma~\ref{lem:Independent_sets} for $H= [1,N] $ and $h = h(N) = \sum_{k \leq N}c_k^2$ to obtain sets $A_0, B_0, C_0$ and   $\xi_0 \in \left[ \frac{1}{4}, \frac{1}{3} \right)$  with $\lvert c_k \rvert \le N_0^ {-\xi_0- \log(N_0)^{-1/2}}$ for all $k \in A_0$, where $N_0 = N$. As long as $N_{j+1}= \# B_j \geq h(N)^{100}$ for $j=0,1,\ldots$, we do the following: We apply Lemma~\ref{lem:Independent_sets} for the choice $H= B_j$ and $h=h(N)$ to obtain sets $A_{j+1}, B_{j+1}, C_{j+1}$ and   $\xi_{j+1} \in \left[\frac{1}{4}, \frac{1}{3} \right)$  with $\forall k \in A_{j+1}: \lvert c_k \rvert \le N_{j+1}^{-\xi_{j+1}- \log(N_{j+1})^{-1/2}}$. We observe 
    \[
    \min \{ | k - \ell | : k \in A_j, \ell \in A_{j+1} \} \ge \exp \left(  \frac{\sqrt{ \log N_{j}}}{2}  \right),
    \]
    which stems from the fact that, by construction, $ \min \{ | k - \ell | : k \in A_j, \ell \in B_{j} \} \ge \exp \left(  \frac{\sqrt{\log N_{j}}}{2}  \right) $. Since $A_j \subseteq B_i$ for all $j > i$, this implies $ \min \{ | k - \ell | : k \in A_i, \ell \in A_{j} \} \ge \exp \left(  \frac{\sqrt{\log N_{i}}}{2}  \right)  $ for all $ 0 \le i < j $, which proves (iv). The claims in (ii) and (iii) follow by construction. 
    \par{}
    Since all $k \in B_j$ satisfy $c_k^2 \ge N_j^{-2 \xi_j}$, the simple relation $\sum_{k \in B_j} c_k^2 \le h(N)$ implies $N_{j+1} \le h(N) N_{j}^{2/3}$ and since $N_j \ge h(N)^ {100}$, we further get $ N_{j+1} \le N_j^ {3/4} \le \ldots \le N^ {(3/4)^j}$. We continue this procedure until we reach an index $T \in \N$ with $ N_T < h(N)^ {100} \le N_{T-1}$ (this happens after $O(\log \log N)$ many steps). This leaves us with a "remainder" set $B_{T-1}=: A_T$ with $\# B_{T-1} = N_T < h(N)^ {100}$. We now set $\Gamma= \dot\bigcup_{j=0}^T A_j$ and we observe that
      \[
    \sum_{k \in [1,N] \setminus \Gamma } c_k^2 \leq \sum_{j=0}^{T-1} \sum_{k \in C_j} c_k^2 \ll h(N) \sum_{j=0}^{T-1} \frac{1}{\sqrt{\log(N_{j})}}.
    \] 
    Since   $N_j \ge N_{j+1}^ {4/3} \ge \ldots \ge N_{T-1}^{(4/3)^ {T-1-j}} \ge h(N)^{(4/3)^{T-1-j}}$, the expression above is clearly of order $o(h(N))$, which shows (i). By construction, we also obtain (v).
    \par{}
    Under the assumption $N < h(N)^{100}$, we set $T=0$, $N_0=N$ and $A_0= [1,N]$, and the properties hold trivially.
\end{proof}

\begin{lem}
\label{lem:Construction_blocks}
  For any $ 0< \gamma < 1$ and any $\eta > 1$, we  can find $M \in \N$ and  non-empty  sets $\Delta_i \subseteq [1,N]$ and $\Delta_i' \subseteq [1,N]$ for $i=1, \ldots, M$ such that  we have:
\begin{itemize}
    \item[(i)] $\Delta_1 < \Delta_1' < \Delta_2 < \ldots $ holds element-wise.
    \item[(ii)] $\sum_{k \in \Delta_i} c_k^2 = O(h(N)^\gamma)$, $\forall i=1, \ldots, M$.
    \item[(iii)] We have $ \sum_{i=1}^M \# \Delta_i^2 \eta^{- \min \{ \# \Delta_i', \# \Delta_{i-1}'\} } = O(1)$, where we formally define $\# \Delta_0'=N$.
    \item[(iv)] We have   $\min_{1 \leq i \leq M}\# \Delta_i' \longrightarrow \infty$  as $N \rightarrow \infty$.
    \item[(v)] $\sum_{k  \in [1,N] \setminus (\Delta_1 \cup \ldots \cup \Delta_M)} c_k^2 = o(h(N))$.
\end{itemize}
\end{lem}

\begin{proof}
Let $K = K(\eta)$ be a large constant determined later.
By Lemma \ref{lem:Gamma} we can write $[1, N] = \Gamma \cup \Gamma^c$ with $\sum_{k \in \Gamma^c} c_k^2 = o(h(N))$ and $\Gamma = \dot\bigcup_{0 \le j \le T} A_j$. 
  For the remainder of this proof, replace $N_T$ by
$\max\{N_T,\lceil \sqrt{h(N)}\rceil\}$, leaving all sets $A_j$
and all other parameters unchanged. We then still have
$\#A_T \leq N_T < h(N)^{100}$ for all sufficiently large $N$,
and additionally $N_T \longrightarrow \infty$.
If $A_T=\emptyset$, no blocks are constructed from this set.
\par{}
These sets $A_j$ might be highly fragmented and interlaced in a complicated manner. However, we can ensure the following: Each $A_j$ breaks down into $A_j = \dot\bigcup_{i=1}^{\widehat{M_j}} \widehat{\Gamma}_{i,j}$, where the   $\widehat{\Gamma}_{1,j} < \widehat{\Gamma}_{2,j} < \ldots < \widehat{\Gamma}_{\widehat{M_j},j}$ are maximal intervals of consecutive integers in $A_j$  such that 
\begin{equation*}
\min \{ | k- \ell | : k \in \widehat{\Gamma}_{i,j}, \ell \in \widehat{\Gamma}_{i', j'} \} \ge K \log(N_j),
\end{equation*}
for   all $0 \le j < j' \le T, 1\leq i \leq \widehat{M_j}, 1\leq i' \leq \widehat{M_{j'}} $. The separation in the above equation is inherited from the separation property of the sets $A_j$ in Lemma~\ref{lem:Gamma}(iv).  
As we will see, this means that the $\widehat{\Gamma}_{i,j},\widehat{\Gamma}_{i',j'}$ have the correct amount of separation if $j \neq j'$. 
However, neither the separation between $\widehat{\Gamma}_{i,j},\widehat{\Gamma}_{i',j}$ for distinct values of $i$ and $i'$, nor the "mass" (in terms of the sum over $c_k^2$) of each $\widehat{\Gamma}_{i,j}$ is under control yet.\\
Let $j=0, \ldots, T$ be fixed. In the first step, starting with $\widehat{\Gamma}_{1,j}$, there exists a neighboring $\widehat{\Gamma}_{i',j'}$ that satisfies
$\widehat{\Gamma}_{1,j} < \widehat{\Gamma}_{i',j'}$ and there exists no $(i'',j'')$ such that 
$\widehat{\Gamma}_{1,j} < \widehat{\Gamma}_{i'',j''} < \widehat{\Gamma}_{i',j'}$. If $j' \neq j$, then we define
$\Gamma_{1,j} := \widehat{\Gamma}_{1,j}$. If $j = j'$, we then continue with $\widehat{\Gamma}_{i',j'}$ in place of $\widehat{\Gamma}_{i,j}$ until after $s_{j,1}$ steps, there are either no elements left or there exists a subsequent interval $\widehat{\Gamma}_{i',j'}$ such that $j' \ne j$. 
In this case, we define $\Gamma_{1,j} = \bigcup_{1 \leq i' \leq s_{j,1}} \widehat{\Gamma}_{i',j}$. 
We continue in the same fashion with $\widehat{\Gamma}_{s_{j,1}+1,j}$ until we obtain a partition
$A_j = \bigcup_{i = 1}^{M_j} \Gamma_{i,j}$
with the following property: if, element-wise, ${\Gamma}_{i,j} < {\Gamma}_{i',j'}$ and there exists no $(i'',j'')$ such that ${\Gamma}_{i,j} < {\Gamma}_{i'',j''} < {\Gamma}_{i',j'}$, then $j \neq j'$. Consequently, we achieve that 
\begin{equation}
\label{eq:distance_gamma}
\min \{ | k- \ell | : k \in \Gamma_{i,j}, \ell \in \Gamma_{i', j'} \} \ge K \log(N_{\min \{ j, j' \} }),
\end{equation}
for all $(i,j) \neq (i',j') $. This means   that  the sets $\Gamma_{i,j}$ are now sufficiently separated from each other and we regard the intervals between consecutive $ \Gamma_{i,j}$ as buffer blocks $\Delta_{\ell}'$. Note that, by construction, the union of these buffer blocks is a subset of $\Gamma^c$ and therefore has negligible mass.
However, the ``mass'' (in terms of the sum over $c_k^2$) of each $\Gamma_{i,j}$ is not yet under control.  We find $r_{i,j} \in \N$ such that $(r_{i,j} -1) h(N)^\gamma \le \sum_{k \in \Gamma_{i,j}} c_k^2 \le r_{i,j} h(N)^\gamma$.
 There are two possible cases: Either   $r_{i,j} = 1$  in which case we leave $\Gamma_{i,j}$ as it is and consider it as a mass block $\Delta_{\ell}$.   Otherwise, we  split $\Gamma_{i,j}$ into at most $r_{i,j}$ sub-sets all having mass of at most $h(N)^ \gamma$, which we additionally separate by buffer blocks of length   $\lceil K \log(N_j)\rceil$  (of which there are at most $r_{i,j}-1$). This is done in   an algorithmic way: stopping exactly before the index where we would have accumulated  $h(N)^{\gamma}$ mass, we consider the following   $\lceil K \log N_j\rceil $  elements of $\Gamma_{i,j}$ as buffer blocks. This will add at most $r_{i,j}-1$ new mass and buffer blocks $\Delta_{\ell}, \Delta_{\ell}'$ to our construction and it will add an error of at most   $ \sum_{i=1}^{M_j}(r_{i,j}-1) \lceil K\log(N_j)\rceil N_j^{-1/2}$  if $0 \le j \le T-1$ (recall that all $c_k \in A_j$ satisfy $c_k^2 \le N_j^{-1/2}$). If $j=T$, we obtain an error that is of order at most $\sum_{i=1}^{M_T} (r_{i,T} -1)K \log(N_T)$.
\par{}
  Order the resulting mass and buffer blocks increasingly.
We discard the final mass block and merge it with the buffer
immediately preceding it and a potential buffer following it, thereby
obtaining a terminal buffer of unbounded length. This removes an additional mass
of at most $O(h(N)^\gamma)=o(h(N))$. After relabelling, we obtain
\[
\Delta_1<\Delta_1'<\cdots<\Delta_M<\Delta_M'.
\]
This  establishes (i). The claim in (ii) follows from our construction.
\par{}
To verify the other claims of the Lemma, we need to estimate the mass emerging from our buffer blocks. There are two types of errors: On the one hand, for buffer blocks within $\Gamma_{i,j}$ for $j=0, \ldots, T-1$ we get at most (recall that $N_j \ge h(N)^{100}$)
\[
\sum_{j=0}^{T-1} \sum_{i=1}^{M_j} (r_{i,j}-1) \log(N_j) N_j^{-1/2} \ll h(N)^{-25} \sum_{j=0}^{T-1} \sum_{i=1}^{M_j} (r_{i,j}-1) = o(h(N)),
\]
where the last estimate follows from the fact that $(r_{i,j} -1) \leq \frac{\sum_{k \in \Gamma_{i,j}} c_k^2}{h(N)^\gamma}$ and since the $\Gamma_{i,j}$ are all disjoint, $\sum_{j=0}^{T-1} \sum_{i=1}^{M_j}(r_{i,j} -1) \le h(N)^{1-\gamma}$. On the other hand, for $j=T$ we pick up an error of at most
\[
\sum_{i=1}^{M_T} (r_{i,T} -1)\log(N_T) \ll \log(h(N)) \sum_{i=1}^{M_T} (r_{i,T} -1) \ll \log(h(N)) h(N)^{1-\gamma}=o(h(N)),
\]
where, analogously to before,   $\sum_{i=1}^{M_T} (r_{i,T}-1) = O(h(N)^{1-\gamma})$. This shows that the overall contribution of the buffer blocks is of order $o(h(N))$. By construction, we thus have 
\[
\sum_{k \in [1,N] \setminus (\Delta_1 \cup \ldots \cup \Delta_M)} c_k^2 \leq \sum_{k \in \Delta_1' \cup \ldots \cup \Delta_M'} c_k^2 + \sum_{k \in \Gamma^c} c_k^2 = o(h(N)),
\]
establishing (v). The claim in (iv) is true since, by construction, all $\Delta_i'$ satisfy $\# \Delta_i' \ge K  \log N_T$ which tends to infinity by assumption. 
\par{}
We are now going to verify (iii). Let $\Delta_i < \Delta_i' < \Delta_{i+1}$ with $\Delta_i \subseteq A_{j_1}$ and $ \Delta_{i+1} \subseteq A_{j_2}$, then there are two possibilities: if $j_1 = j_2$, then   $\Delta_i \cup \Delta_{i+1} \subseteq \Gamma_{i',j}$
 for some $(i',j')$, thus by the added buffer blocks within   $\Gamma_{i',j}$, we get $\# \Delta_i' \ge K \log(N_{j_1})$. If $j_1 \neq j_2$, then by \eqref{eq:distance_gamma} we have $ \# \Delta_i' \ge K \log(N_{\min \{j_1, j_2\} })$. 
\par{}
For $j=0, \ldots, T$, let $K_j = \{ 1 \le i \le M : \Delta_i \subseteq A_j \}$. Then the sets $K_j$ are disjoint and we clearly have $\# K_j \le N_j$ for all possible values of $j$. For $i \in K_j$, by construction we have $\# \Delta_i \le N_j$ and by the previous observation $\min \{ \# \Delta_i', \# \Delta_{i-1}' \} \ge K \log N_j $. This leads to
  \[
\sum_{i=1}^M \# \Delta_i^2 \eta^{- \min \{\# \Delta_{i-1}', \# \Delta_i' \}} \le \sum_{j =0}^T \sum_{i \in K_j} N_j^2 \eta^{-K \log N_j} \leq \sum_{j=0}^T N_j^{- K \log \eta +3 }.
\] 
Since   $N_j \ge (N_{j+1})^{4/3} \ge \ldots \ge (h(N)^{100})^{(4/3)^{T-1-j}}$  for any $j \le T-1$, by choosing $K$ sufficiently large, the latter sum is clearly bounded from above by
\[
N_T^{-1} + \sum_{j=0}^{T-1} h(N)^{-j} = O(1).
\]
This finishes the proof.
\end{proof}

Given the construction of blocks $\Delta_i <\Delta_i'< \Delta_{i+1}\ldots$ from before, we can define a suitable filtration $(\mathcal{F})_{i \le M}$ and an adapted process $(Y_i)_{i \le M}$ which we will later use to apply Proposition \ref{prop:martingale_clt}. In the following, let $\alpha_i = \min \{ \Delta_{i} \}$ and $ \beta_i = \max \{ \Delta_i \}$.

\begin{lem}
\label{lem:new_filtration}
Let $p$ be a trigonometric polynomial  with mean zero and  of fixed degree $d$, and fix $\gamma \in (0,1)$ and choose $\eta = q^{1/4} > 1$. Let $\Delta_1 < \Delta_1' < \ldots$ be the blocks from Lemma~\ref{lem:Construction_blocks}. For $k \in \Delta_i$ we define 
\[
m(k) = \left \lceil \log_2 n_{k} + \frac{1}{J}\min \{ \#\Delta_i', \# \Delta_{i-1}' \} \right \rceil,
\]
where $J = \frac{2}{\log_2 q}$. Further, we set
\[
\mathcal{G}_k = \sigma\left(\left\{ \left[ \frac{v}{2^{m(k)}},\frac{v+1}{2^{m(k)}} \right): v = 0,\ldots, 2^{m(k)}-1\right\}\right), \quad \F_i = \mathcal{G}_{\beta_i}, \quad  \text{as well as } \quad \F_0 = \{ \T, \emptyset \}.
\]
For every $k \in \Delta_i$, there exists a function $\varphi_k: \T \rightarrow \R$ such that the following holds:

\begin{itemize}
    \item[(i)] $\varphi_k$ is constant on $I_{v,k} = \left[\frac{v}{2^{m(k)}},\frac{v+1}{2^{m(k)}}\right)$ for any $v = 0,\ldots ,2^{m(k)}-1$.
    \item[(ii)] $\lVert\varphi_k - p(n_k \cdot)\rVert_{\infty} \ll  2^{-\frac{1}{J} \min \{ \# \Delta_i', \# \Delta_{i-1}'\}}$.
    \item[(iii)] $\mathbb{E}[\varphi_k\vert\mathcal{F}_{i-1}] \equiv 0$.
\end{itemize}
\end{lem}

\begin{proof}
 The case $i = 1$ is trivial, so we may assume $i \geq 2$.
 For $k \in \Delta_i$, we define $\hat{\varphi}_k = \mathbb{E} \left[ p(n_k \cdot) | \mathcal{G}_k \right]$ and we note that (i) holds for $\hat{\varphi}_k$. For any $x,x' \in I_{v,k}$, Lipschitz continuity of $p$ implies
\[
\left| p(n_k x) - p(n_k x')  \right| \ll \frac{n_k}{2^{m(k)}} \ll \frac{1}{2^{\frac{1}{J} \min \{ \# \Delta_{i-1}', \# \Delta_i'\} }}.
\]
This proves (ii) for the approximations $\hat{\varphi}_k$. We now define
\[
\varphi_k = \hat{\varphi}_k - \mathbb{E}\left[ \hat{\varphi}_k | \mathcal{F}_{i-1} \right].
\]
Then (i) immediately carries over to $\varphi_k$, since $\mathcal{F}_{i-1} \subseteq \mathcal{G}_k$, as ensured by the construction of $\mathcal{G}_k$. In order to prove (ii), we observe that by Lemma \ref{riemann_lebesgue} (we set $\# \Delta_{-1} = N$, which is consistent with our definition of $\F_0$)
  \begin{align*}
\lVert \varphi_k - p(n_k \cdot ) \rVert_\infty 
& \leq \left \lVert \mathbb{E} \left[ \hat{\varphi}_k | \mathcal{F}_{i-1}  \right] \right \rVert_\infty + \left \lVert  \hat{\varphi}_k  - p(n_k \cdot) \right \rVert_\infty \\
& \ll \max_{v=0, \ldots, 2^{m(\beta_{i-1})}-1} \frac{1}{|I_{v,\beta_{i-1}}|} \left| \int_{I_{v,\beta_{i-1}}} \hat{\varphi}_k(t) \mathrm{d}t \right|  + \frac{1}{2^{\frac{1}{J} \min \{ \# \Delta_{i-1}', \# \Delta_i'\} }}\\
& \ll \max_{v=0, \ldots, 2^{m(\beta_{i-1})}-1} \frac{1}{|I_{v,\beta_{i-1}}|} \left| \int_{I_{v,\beta_{i-1}}} p(n_k t) \mathrm{d}t \right| + \frac{1}{2^{\frac{1}{J} \min \{ \# \Delta_{i-1}', \# \Delta_i'\} }} \\
& \ll \frac{2^{m(\beta_{i-1})}}{n_k} + \frac{1}{2^{\frac{1}{J} \min \{ \# \Delta_{i-1}', \# \Delta_i'\} }} \\
& \ll \frac{n_{\beta_{i-1}}2^{\frac{1}{J} \min \{ \# \Delta_{i-2}', \# \Delta_{i-1}'\} }}{n_{\alpha_i}} + \frac{1}{2^{\frac{1}{J} \min \{ \# \Delta_{i-1}', \# \Delta_i'\} }} \\
& \ll \left( \frac{q}{2^ {\frac{1}{J}}} \right)^{-\# \Delta_{i-1}'} + 2^{-\frac{1}{J} \min \{ \# \Delta_{i-1}', \# \Delta_i'\} } \\
& \ll 2^{-\frac{1}{J} \min \{ \# \Delta_{i-1}', \# \Delta_i'\} },
\end{align*} 
where the final estimate follows from the choice of $J= \frac{2}{\log_2(q)}$.
\end{proof}
\begin{rem}
In the proof of Lemma \ref{lem:new_filtration} we have seen that $\mathcal{F}_{i-1} \subseteq  \mathcal{G}_k $ for all $k \in \Delta_i$ and thus $\mathcal{F}_{i-1} \subseteq \mathcal{F}_i$, which shows that $(\mathcal{F}_i)_{i =0}^M$ is indeed a filtration.
\end{rem}

Next, we use the random variables $(\varphi_k)_{k \in \Delta_i}$ to construct a suitable martingale. For $1 \leq i \leq M$ we define
\begin{equation}
\label{eq:Yi_Ti_VM}
Y_i = \sum_{k \in \Delta_i}c_k\varphi_k, \quad T_i = \sum_{k \in \Delta_i}c_kp(n_k \cdot ),\quad 
V_M = \sum_{i = 1}^M \mathbb{E}[Y_i^2\vert\F_{i-1}].
\end{equation}
By construction, $(Y_i)_{i =1}^M$ is adapted to $(\mathcal{F}_i)_{i =1}^M$, and since $\mathbb{E}[Y_i | \mathcal{F}_{i-1}] = 0$,
  For each fixed $N$, the partial sums
$X_m=\sum_{i=1}^mY_i$,   $0\leq m\leq M$, form a martingale.
By formally extending this martingale by setting $Y_i=0$ and
$\mathcal F_i=\mathcal F_M$ for $i>M$, we may apply
 Proposition~\ref{prop:martingale_clt}   at its terminal index $M$.
 With this in mind, we  prove the following estimate:
\begin{lem}
\label{lem:est_Ti_Yi}
   In the above notation, we have
      \[
    \sum_{i=1}^M \lVert T_i^2 - Y_i^2\rVert_{\infty} = O(1)
    \] 
     as well as
    \[
    \sum_{i=1}^M \lVert  T_i - Y_i \rVert_{\infty} = O(1).
    \]
 \end{lem}
\begin{proof}
For any $i \in \{1, \ldots, M \}$ we have
    \[
    \begin{split}
    \lVert T_{i}^2 - Y_i^2\rVert_{\infty} &= \sup_{x \in \T} \left |\left(\sum_{k \in \Delta_{i}}\left(c_k(p(n_kx) - \varphi_k(x))\right)\right)\left(\sum_{k \in \Delta_i}c_k(p(n_kx) + \varphi_k(x))\right)\right|
    \\
    &\ll \# \Delta_i \sum_{k \in \Delta_i} \lvert c_k \rvert \sup_{x \in \T}\lvert p(n_kx) - \varphi_k(x)\rvert
    \\
    &\ll \# \Delta_i^2 2^{-\frac{1}{J} \min \{ \# \Delta_i', \#  \Delta_{i-1}' \} },
    \end{split}
    \]
where we have used Lemma~\ref{lem:new_filtration}(ii). By Lemma~\ref{lem:Construction_blocks}(iii) the  first  claim follows.  For the second claim, we use
\[\lVert T_{i} - Y_i\rVert_{\infty} = \sup_{x \in \T}  \left|\left(\sum_{k \in \Delta_{i}}\left(c_k(p(n_kx) - \varphi_k(x))\right)\right)\right|
    \\
    \ll \# \Delta_i \max_{k \in \Delta_i}\sup_{x \in \T}\lvert p(n_kx) - \varphi_k(x)\rvert,\]
    and conclude as before.
 \end{proof}

Let us further define the quantities
\begin{equation}
\label{Eq:w_i_s_M}
\quad w_i = \int_{\T} Y_i(x)^2 \,\mathrm{d}x, \quad s_M = \sum_{i =1}^M w_i.
\end{equation}
The following lemma is the final key ingredient in order to apply Proposition~\ref{prop:martingale_clt}.
\begin{lem}
\label{lem:est_VM_sM}
Assume that $p(x) = \sum_{j=1}^d a_j \cos(2 \pi j x)$, $(n_k)_{k \in \N}$ and $(c_k)_{k \le N \in \N}$ either satisfy Assumption~\ref{Ass:dioph_cond_dream} or Assumption~\ref{Ass:dioph_cond_homog}. In the above notation, we have
    \begin{equation}
        \lVert V_M - s_M\rVert_2^2 = o(h(N)^2),
    \end{equation}
    where the implied constant only depends on $p$.
\end{lem}

\begin{proof}
Using Lemma~\ref{lem:est_Ti_Yi} we obtain
  \begin{equation}
\label{eq:first_est_VM_sM}
\begin{split}
   \left \lVert V_M - s_M \right \rVert_2 
    &=
   \left \lVert \sum_{i = 1}^M \E[Y_i^2|\F_{i-1}] - s_M \right \rVert_2 \\
    & \leq \left \lVert \sum_{i = 1}^M \E[T_i^2|\F_{i-1}] - s_M \right \rVert_2 + O(1) \\
    &= \left \lVert \sum_{i = 1}^M \E[T_i^2 - \E[T_i^2]|\F_{i-1}] \right \rVert_2 + O(1).
\end{split}
\end{equation} 
We now examine $T_i^2 - \E[T_i^2]$ more closely. 
 To that end, we recall the trigonometric identity \[
\cos(x) \cos(y) = \frac{1}{2} \left( \cos(x-y) + \cos(x+y) \right).
\]
In the following, we define $n_{\beta_0} = 0$ and $ \# \Delta_{-1}' = \# \Delta_0' = N$. Let $\rho_i = 2^{\frac{1}{2J} \min \{ \# \Delta_{i-1}', \#  \Delta_{i-2}' \} } $, then

\begin{align*}
T_i^2(x)-\E[T_i^2]
&= \frac{1}{2} \left(
\sum_{\substack{
k,k'\in\Delta_i,\;1\le j,j'\le d\\
0<|jn_k-j'n_{k'}|<\rho_i n_{\beta_{i-1}}
}}
c_k c_{k'} a_j a_{j'}
\cos\bigl(2\pi(jn_k-j'n_{k'})x\bigr)
\right)\\
&\quad+\frac{1}{2} \left(
\sum_{\substack{
k,k'\in\Delta_i,\;1\le j,j'\le d\\
\rho_i n_{\beta_{i-1}}\le |jn_k-j'n_{k'}|<n_{\alpha_i}
}}
c_k c_{k'} a_j a_{j'}
\cos\bigl(2\pi(jn_k-j'n_{k'})x\bigr)
\right)\\
&\quad+\frac{1}{2} \left(
\sum_{\substack{
k,k'\in\Delta_i,\;1\le j,j'\le d\\
n_{\alpha_{i}}\le |jn_k-j'n_{k'}|
}}
c_k c_{k'} a_j a_{j'}
\cos\bigl(2\pi(jn_k-j'n_{k'})x\bigr)
\right)\\
&\quad+\frac{1}{2} \left(
\sum_{\substack{
k,k'\in\Delta_i,\;1\le j,j'\le d
}}
c_k c_{k'} a_j a_{j'}
\cos\bigl(2\pi(jn_k+j'n_{k'})x\bigr)
\right)\\
&= U_i(x)+ W_i(x)+\underbrace{R_i^{(1)}(x)+R_i^{(2)}(x)}_{= R_i(x)},
\end{align*} 
where in the first estimate, we employed Lemma~\ref{lem:est_Ti_Yi} to ensure that the "diagonal" term (where $jn_k -j' n_{k'} = 0$) cancels out. This means, for the quantity in~\eqref{eq:first_est_VM_sM}, we get
  \begin{equation}
\label{eq:VM_sM_split_Ui_Wi}
\left \lVert V_M -s_M \right \rVert_2 \ll \left \lVert \sum_{i = 1}^M \E[U_i |\F_{i-1}] \right \rVert_2 +  \left \lVert \sum_{i = 1}^M \E[W_i|\F_{i-1}] \right \rVert_2 +\left \lVert \sum_{i = 1}^M \E[R_i|\F_{i-1}] \right \rVert_2 + O(1).
\end{equation} 
We see that $R_i$ consists of $ \ll | \Delta_i|^2$ cosine terms, each having frequency at least $n_{\alpha_i}$. Using an analogous argument as in the proof of Lemma~\ref{lem:est_Ti_Yi}, we obtain
\begin{equation}
\label{eq:estimate_R}
\sum_{i=1}^M\left \lVert \mathbb{E} \left[ R_i | \F_{i-1} \right] \right \rVert_{\infty} = o(h(N)),
\end{equation}
which implies
\[
\left \lVert \sum_{i = 1}^M \E[R_i|\F_{i-1}] \right \rVert_2 = o(h(N)).
\]
Next, we examine $\left\lVert \sum_{i=1}^M \E[U_i|\F_{i-1}]\right\rVert_2$. Since each $U_i$ is a sum of cosines with frequencies ranging from $1$ to $ \rho_i n_{\beta_{i-1}}$, we can write
\begin{align*}
U_i(x) = \sum_{u=1}^{\lceil \rho_i n_{\beta_{i-1}} \rceil-1 } \bar{c}_{u,i} \cos(2 \pi ux)
\end{align*}
where
\[
\bar{c}_{u,i} = \sum_{\substack{j,j' \leq d \\ k,k' \in \Delta_i}}c_kc_{k'}a_ja_{j'}\mathds{1}_{[jn_k - j'n_{k'} = u]}.
\]
We obtain (note that the sequence $(\rho_i n_{\beta_{i-1}})_{i=1,2, \ldots}$ is strictly monotonically increasing)
\begin{align*}
\sum_{i = 1}^M U_i(x) & = \sum_{i=1}^M \sum_{u = 1}^{\lceil \rho_i n_{\beta_{i-1}} \rceil -1 } \bar{c}_{u,i} \cos(2 \pi ux) \\
&= \sum_{u = 1}^{\lceil \rho_M n_{\beta_{M-1}} \rceil-1} \cos(2\pi u x) \underbrace{\sum_{i=k_u}^{M} \bar{c}_{u,i}}_{=d_u},
\end{align*}
where $k_u \in \{1, \ldots, M \} $ is a uniquely determined index that arises when exchanging the order of summation.
By Parseval's identity we derive
  \begin{equation}
\label{sup_sum}
\left\lVert \sum_{i=1}^M U_i\right\rVert_2^2 = \frac{1}{2}\sum_{u = 1}^{\lceil \rho_M n_{\beta_{M-1}} \rceil-1} d_u^2 \leq 
\left(\sup_{v \in \N }\lvert d_v\rvert\right)\cdot  \sum_{u=1}^{\lceil \rho_M n_{\beta_{M-1}} \rceil-1}\lvert d_u\rvert.
\end{equation} 
In the following, we first estimate $\sup_{v \in \N} | d_v|$, and then the sum over all $|d_u|$. Note that
  \begin{align*}
\lvert d_v \rvert 
& \leq \sum_{i=1}^M \lvert \bar{c}_{v,i} \rvert \\
& \leq \sum_{j,j' \leq d}\lvert a_ja_{j'}\rvert  \sum_{\substack{k,k' \leq N}}|c_k| |c_{k'}|\mathds{1}_{[jn_k - j'n_{k'} = v]} \\
& \leq 2 \lVert p \rVert_2^2  \sum_{j,j' \leq d}   \sum_{\substack{k,k' \leq N}}|c_k| |c_{k'}|\mathds{1}_{[jn_k - j'n_{k'} = v]} \\
& \leq 2 \lVert p \rVert_2^2 L(N,d) \\
& = o(h(N)),
\end{align*} 
where we have used the elementary estimate $| a_j a_{j'} | \leq 2 \lVert p \rVert_2^2$ and the fact that $L(N,d)=o(h(N))$ which follows from Assumption~\ref{Ass:dioph_cond_dream} or Assumption~\ref{Ass:dioph_cond_homog} respectively.
Thus, we infer
\begin{equation}
\label{eq:est_max_du}
\sup_{v \in \N } | d_v | = o( h(N)),
\end{equation}
and continue with the estimate for the sum in \eqref{sup_sum}. By switching the order of summation, we obtain
\[
\sum_{u=1}^{\lceil \rho_M n_{\beta_{M-1}} \rceil-1} | d_u|
\leq \sum_{i=1}^M \sum_{u=1}^{\lceil \rho_i n_{\beta_{i-1}} \rceil-1} \lvert \bar{c}_{u,i}\rvert.
\]
Using $\rho_i n_{\beta_{i-1}} \leq n_{\alpha_i}$, which holds by the definition of $\alpha_i$ and $\beta_{i-1}$, we further get
  \begin{align*}
\sum_{i=1}^M \sum_{u=1}^{\lceil \rho_i n_{\beta_{i-1}} \rceil-1} \lvert \bar{c}_{u,i}\rvert
& \leq \sum_{i=1}^M \sum_{j,j' \leq d}\lvert a_ja_{j'}\rvert  \sum_{\substack{k,k' \in \Delta_i \\
\lvert jn_k - j'n_{k'}\rvert \leq n_{\alpha_i}}}|c_k||c_{k'}|  \\
& \leq 4 \lVert p \rVert_2^2 \sum_{i=1}^M \sum_{j,j' \leq d} \sum_{\substack{k\leq k' \in \Delta_i \\
\lvert jn_k - j'n_{k'}\rvert \leq n_{\alpha_i}}}|c_k||c_{k'}| .
\end{align*} 
Observe that if $k \leq k'$ then $j'n_{k'} - jn_k \geq n_{k'} - dn_k \geq n_{k}q^{k'-k} - dn_k$. Thus, for   $k' \geq k + \log_q(3d)$, we have
  $\lvert j'n_{k'} - jn_k\rvert \geq 3dn_k - dn_k > n_{\alpha_i}$. Therefore it follows that
  \begin{equation}
\label{eq:sum_ck_ckprime}
\sum_{\substack{k\leq k' \in \Delta_i \\
\lvert jn_k - j'n_{k'}\rvert \leq n_{\alpha_i}}}|c_k||c_{k'}|
\leq \sum_{k \in \Delta_i}\sum_{\substack{k \leq k' \leq k + \log_q(3d)\\k' \in \Delta_i}} |c_k||c_{k'}| 
\ll \sum_{k \in \Delta_i}c_k^2.
\end{equation} 
Thus,
  \[
\sum_{u=1}^{\lceil \rho_M n_{\beta_{M-1}} \rceil-1} \lvert d_u\rvert \leq \sum_{i=1}^M \sum_{u=1}^{\lceil \rho_i n_{\beta_{i-1}} \rceil-1} \lvert \bar{c}_{u,i}\rvert \ll \lVert p \rVert_2^2 \sum_{i=1}^M \sum_{j,j' \leq d} \sum_{k \in \Delta_i}c_k^2 \ll h(N).
\] 
Combining the estimate above with~\eqref{eq:est_max_du} now yields by~\eqref{sup_sum} that 
\[
\left\lVert \sum_{i=1}^M U_i\right\rVert_2^2 \leq \sum_{u = 1}^{\lceil \rho_M n_{\beta_{M-1}} \rceil-1} d_u^2 \leq 
\sup_{v \in \N} \lvert d_v \rvert\sum_{u=1}^{\lceil \rho_M n_{\beta_{M-1}} \rceil-1}\lvert d_u\rvert = o(h(N)^2).
\]
As a byproduct of the previous argument, it follows that for any $i=1, \ldots, M$, we have 
\begin{equation}\label{byproduct}
\sum_{u=1}^{\lceil \rho_i n_{\beta_{i-1}} \rceil-1} \lvert \bar{c}_{u,i}\rvert \ll h(N)^\gamma. 
\end{equation}
In order to estimate $  \sum_{i=1}^M \left \lVert \mathbb{E} \left[ U_i | \F_{i-1} \right] - U_i\right \rVert_2$, we observe that the fluctuations of $U_i$ over any atom of $\F_{i-1}$ are bounded above by
\begin{align*}
\sum_{u=1}^{\lceil \rho_i n_{\beta_{i-1}} \rceil-1} | \bar{c}_{u,i}| \left | 2 \pi u 2^{-m(\beta_{i-1})} \right| & \ll \rho_i \frac{n_{\beta_{i-1}}}{2^{m(\beta_{i-1})}}  \underbrace{\sum_{u=1}^{\lceil \rho_i n_{\beta_{i-1}} \rceil-1} | \bar{c}_{u,i} |}_{\ll h(N)^\gamma \text{ by } \eqref{byproduct}} \\
& \ll \rho_i^{-1} h(N)^\gamma,
\end{align*}
where the last step employs the definition of $\rho_i$. 
Note that Lemma~\ref{lem:Construction_blocks}(iii) implies
\begin{equation}\label{rho_summ}\sum_{i =1}^M \rho_i^{-1} = O(1).
\end{equation}
Therefore, we obtain
\[
\sum_{i=1}^M \left \lVert \mathbb{E} \left[ U_i | \F_{i-1} \right] -U_i \right \rVert_2 = O(h(N)^{\gamma}).
\]
Together with the estimate on $ \left \lVert \sum_{i=1}^M U_i \right \rVert_2^2$ this implies
\begin{align}
\label{eq:est_Ui} 
\left \lVert   \sum_{i=1}^M \mathbb{E} \left[ U_i | \F_{i-1} \right] \right \rVert_{2} 
&\leq \left \lVert   \sum_{i=1}^M  U_i  \right \rVert_{2} + \sum_{i=1}^M \left  \lVert   \mathbb{E}[U_i | \mathcal{F}_{i-1}] - U_i   \right \rVert_{2}= o( h(N)).
\end{align}
Next, we estimate $\left\lVert\sum_{i=1}^M \E[W_i|\F_{i-1}]\right\rVert_2$. We note that
\[
\left\lVert\sum_{i=1}^M \E[W_i|\F_{i-1}]\right\rVert_2^2 = \E\left[\sum_{1 \leq i,i'\leq M} \E[W_i|\F_{i-1}]\E[W_{i'}|\F_{i'-1}]\right].
\]
The diagonal terms where $i=i'$ yield
\[
\E\left[\sum_{i = 1}^M \E[W_i|\F_{i-1}]^2\right].
\]
Since   $W_i(x) = \sum_{u=\lceil \rho_i n_{\beta_{i-1}} \rceil }^{ n_{\alpha_i}-1} \bar{c}_{u,i} \cos(2\pi u x)$, Jensen's inequality for the conditional expectation and the same argument as in the lines above~\eqref{eq:sum_ck_ckprime}   reveal  that
\begin{equation}
\label{eq:est_Wi_diag}
\begin{split}
\E\left[\sum_{i=1}^M \E[W_i|\F_{i-1}]^2\right] & \leq \sum_{i=1}^M\lVert W_i \rVert_{2}^2 \\
& \ll \sum_{i = 1}^M\left( \sum_{k \in \Delta_i}c_k^2 \right)^2 \\
& \ll h(N)^\gamma \sum_{i=1}^M \sum_{k \in \Delta_i} c_k^2 \\
& = o(h(N)^2),
\end{split}
\end{equation}
where we have used the estimate $\sum_{k \in \Delta_i} c_k^2 = O(h(N)^\gamma)$ in the second to last line.

We are therefore left to treat the case where $i < i'$. Since $\E[W_i|\F_{i-1}]$ is $\F_{i-1}$-measurable, we get

  \[
\begin{split}
\left\lvert \E\left[\E[W_i|\F_{i-1}]\E[W_{i'}|\F_{{i'}-1}] \big\vert \F_{i-1}\right]\right\rvert
& = \left\lvert\E[W_i|\F_{i-1}]\E[W_{i'} \big\vert \F_{i-1}]\right\rvert \\
& \leq \lVert W_i\rVert_{\infty}\left\lvert\E[W_{i'} \big\vert \F_{i-1}]\right\rvert \\
& \ll \left( \sum_{k \in \Delta_i}c_k^2 \right) \left\lvert\E[W_{i'} \big\vert \F_{i-1}]\right\rvert \\
& \ll h(N)^\gamma \left\lvert\E[W_{i'} \big\vert \F_{i-1}]\right\rvert.
\end{split}
\] 
The second to last inequality utilizes an analogous argument to ~\eqref{eq:sum_ck_ckprime}.

Write   $W_{i'}(x) = \sum_{u=\lceil\rho_{i'} n_{\beta_{i'-1}} \rceil}^{ n_{\alpha_{i'}}-1} \bar{c}_{u,i'} \cos(2\pi u x)$  and recall that
  \[
\sum_{u=\lceil \rho_{i'} n_{\beta_{i'-1}} \rceil}^{ n_{\alpha_{i'}}-1} \lvert \bar{c}_{u,i'} \rvert \ll \sum_{k \in \Delta_{i'}}c_k^2 \ll h(N)^\gamma.
\] 
By Lemma~\ref{riemann_lebesgue} for $f(x) = \cos(2\pi x)$ together with the estimate above, we obtain
  \begin{equation*}
\begin{split} 
\left\lvert\E[W_{i'} \big\vert \F_{i-1}]\right\rvert 
& \leq \sum_{u=\lceil  \rho_{i'} n_{\beta_{i'-1}} \rceil }^{ n_{\alpha_{i'}}-1} \lvert \bar{c}_{u,i'} \rvert u^{-1}2^{m(\beta_{i-1})} \\
& \leq \frac{2^{m(\beta_{i-1})}}{\rho_{i'}n_{\beta_{i'-1}} } \sum_u \lvert \bar{c}_{u,i'}\rvert \\
& \ll \frac{2^{m(\beta_{i-1})}}{\rho_{i'} n_{\beta_{i'-1}} } h(N)^\gamma \\
& \ll (\rho_{i'})^{-1} \frac{n_{\alpha_i}}{n_{\beta_{i'-1}}} h(N)^\gamma,
\end{split}
\end{equation*} 
  where in  the last step, we have applied the definition of ${m(\beta_{i-1})}$.   It follows that

\begin{equation}
\label{eq:est_Wi_offdiag}
\begin{split}
    \sum_{i=1}^{M-1} \sum_{i'=i+1}^M \left\lvert \E\left[\E[W_i|\F_{i-1}]\E[W_{i'}|\F_{{i'}-1}] \big\vert \F_{i-1}\right]\right\rvert & \ll 
    h(N)^{2 \gamma} \sum_{i=1}^{M-1}   \sum_{i'=i+1}^M \rho_{i'}^{-1}\frac{n_{\alpha_i}}{n_{\beta_{i'-1}}}
    \\& = h(N)^{2 \gamma}\sum_{i' = 2}^{M} \rho_{i'}^{-1} \frac{1}{n_{\beta_{i'-1}}}\underbrace{\sum_{i = 1}^{i'-1}n_{\alpha_i}}_{\ll n_{\alpha_{i'-1}}}
    \\& \ll h(N)^{2 \gamma}\sum_{i' = 1}^{M} \rho_{i'}^{-1} = O(h(N)^{2\gamma}),
\end{split}
\end{equation} 
where we use~\eqref{rho_summ} in the   final step.
 Combining~\eqref{eq:est_Wi_diag} and~\eqref{eq:est_Wi_offdiag} we obtain
\begin{align*}
\left\lVert\sum_{i=1}^M \E[W_i|\F_{i-1}]\right\rVert_2^2
& =  \E\left[\sum_{1 \leq i\leq M} \E[W_i|\F_{i-1}]^2\right] + 2\E\left[\sum_{1 \leq i<i'\leq M} \E[W_i|\F_{i-1}]\E[W_{i'}|\F_{i'-1}]\right] \\
& = o(h(N)^2).
\end{align*}
This finishes our analysis of the terms $(W_i)_{i=1}^M$.   Combining  the previous estimate on the $W_i$   with  the estimates in~\eqref{eq:estimate_R} and ~\eqref{eq:est_Ui}, we finally obtain
\[
\lVert V_M - s_M\rVert_2^2 = o(h(N)^2),
\]
which completes the proof.
\end{proof}

The following straightforward lemma will be used several times in the sequel.
\begin{lem}
\label{semitriv_lem}
 Let $1 \leq j,j' \leq d$   be fixed. Then, for any finite set $\Delta \subseteq [1,N]$ that
 \[
\sup_{c \geq 0 }  \sum_{k,k' \in \Delta} \left \lvert c_k c_{k'} \right \rvert \mathds{1}_{[jn_k - j'n_{k'} = c]}  \leq \sum_{k \in \Delta}c_k^2.
 \] 
\end{lem}

\begin{proof}
For fixed $1 \leq j,j' \leq d$,   $c \geq 0$  and for fixed $1 \leq k \leq N$, there is at most one $k'$ such that $j n_k - j' n_{k'} = c$ and these $k'$ are clearly distinct for different $k$. Thus this defines an injection $k'=k'(k)$ from   $ M_{\Delta} = M_{\Delta}(c)= \{k \in \Delta: \exists k' \in \Delta: j n_k - j' n_{k'} = c\}$ to $ \Delta$. We obtain
  \begin{align*}
  \sum_{k,k' \in \Delta} \left \lvert c_k c_{k'} \right \rvert \mathds{1}_{[jn_k - j'n_{k'} = c]}  & \le
 \sum_{k \in M_{\Delta}} \underbrace{ \lvert c_k c_{k'(k)} \rvert}_{\leq \frac{1}{2}(c_k^2 + c_{k'(k)}^2)} \underbrace{\mathds{1}_{[jn_k - j'n_{k'(k)} = c]}}_{\leq 1}   \\
& \leq \sum_{k \in \Delta} c_{k}^2.
\end{align*} 
Since the estimate is uniform in   $c \geq 0$, the claim follows.
\end{proof}
Proposition~\ref{prop:martingale_clt} requires a bound on $\sum_{i=1}^M \mathbb{E} \left[ Y_i^4 \right]$. This is achieved in the following auxiliary result. Similar estimates have been obtained in~\cite{aistleitner_lil_2010_II} and~\cite[Lemma 2.2]{BF_1979_invariance} under slightly more restrictive assumptions.

 \begin{prop}
 \label{prop:fourth_moment}
 Let $(Y_i)_{i=1}^M$ be the sequence of random variables defined in Equation~\eqref{eq:Yi_Ti_VM}. Then we have
\[
\sum_{i = 1}^M \E[Y_i^4] = o( h(N)^{2}),
\]
where the implied constant only depends on $p$.
\end{prop}

 \begin{proof}
Since   $ \lVert T_i^4-Y_i^4 \rVert_1 \le \lVert T_i^2 + Y_i^2 \rVert_1 \cdot \lVert T_i^2 - Y_i^2 \rVert_\infty $, Lemma~\ref{lem:est_Ti_Yi} and Lemma~\ref{semitriv_lem} imply
\[
\begin{split}
 \sum_{i=1}^M \mathbb{E} \left[ Y_i^4 \right] = \sum_{i=1}^M \mathbb{E} \left[ T_i^4 \right] + O(h(N)^\gamma).
\end{split}
\]
Recall that
\[
T_i(x) = \sum_{k \in \Delta_i} c_k p( n_k x),
\]
where $ p(n_k x) = \sum_{j=1}^d a_j \cos(2 \pi j n_k x) $. Thus, taking the $4$-th power of $T_i$ and   using  trigonometric identities such as   $\cos( \alpha)\cos(\beta) = \frac{1}{2} \left( \cos(\alpha - \beta)  +  \cos(\alpha + \beta) \right)$, we end up with a large sum of cosines with respective frequencies of the form $\pm j_1 n_{k_1} \pm \ldots \pm j_4 n_{k_4}$, where $1 \leq j_\ell \leq d $ and $k_\ell \in \Delta_i$.
Since we are only interested in $\E \left[ T_i^4 \right]$, only those summands potentially contribute for which $\pm j_1 n_{k_1} \pm \ldots \pm j_4 n_{k_4}=0$. This observation reveals
\[
\begin{split}
\E[T_i^4] &\leq \!\!\!\sum_{k_1,k_2,k_3,k_4 \in \Delta_i}\!\!\!\!\!\!\lvert c_{k_1}c_{k_2}c_{k_3}c_{k_4} \rvert \!\!\!
\sum_{1 \leq j_1,j_2,j_3,j_4 \leq d}\!\!\!\!\!\!\vert a_{j_1}a_{j_2}a_{j_3}a_{j_4}\vert\mathds{1}_{[ \pm j_1n_{k_1}\pm j_2n_{k_2}\pm j_3n_{k_3}\pm j_4n_{k_4} = 0]} \\
&\ll \!\!\!\!\!\!\sum_{1 \leq j_1,j_2,j_3,j_4 \leq d}\!\!\!\!\!\!\!\!\!\!\!\!\vert a_{j_1}a_{j_2}a_{j_3}a_{j_4}\vert \!\!\!\!\!\!\!\!\!\sum_{k_1\leq k_2 \leq k_3 \leq k_4 \in \Delta_i}\!\!\! \lvert c_{k_1}c_{k_2}c_{k_3}c_{k_4} \rvert \mathds{1}_{[\pm j_1n_{k_1}\pm j_2n_{k_2}\pm j_3n_{k_3}\pm j_4n_{k_4} = 0]}.
\end{split}
\]
Let $0 = \varepsilon_1 j_1n_{k_1}+ \varepsilon_2 j_2n_{k_2}+ \varepsilon_3 j_3n_{k_3}+ \varepsilon_4 j_4n_{k_4}$ for some $\varepsilon_i \in \{-1,1\}$. By multiplying with $\varepsilon_4$, we get
\[
0 =j_4n_{k_4} \pm j_1n_{k_1}\pm j_2n_{k_2}\pm j_3n_{k_3} \geq n_{k_4} - 3dn_{k_3}.
\]
This implies $n_{k_3} \geq \frac{n_{k_4}}{3d}$ and thus by the Hadamard-gap condition we have $k_3 \geq k_4 - \lceil \log_q(3d)\rceil$. Furthermore, for fixed $(k_3,k_4)$, for every $k_2$, there exists at most one $k_1$ such that the equation   $\varepsilon_1 j_1n_{k_1}+ \varepsilon_2 j_2n_{k_2}+ \varepsilon_3j_3n_{k_3} + \varepsilon_4j_4n_{k_4} = 0$  is satisfied, and such $k_1$ are distinct for different $k_2$. This allows us to construct an injection $ k_2 \mapsto {k_{1,k_3,k_4}}(k_2)$ which further implies 

  \[
\begin{split}
&\hspace{-7mm}\sum_{k_1\leq k_2 \leq k_3 \leq k_4 \in \Delta_i} \lvert c_{k_1}c_{k_2}c_{k_3}c_{k_4} \rvert \mathds{1}_{[\varepsilon_1j_1n_{k_1}+\varepsilon_2 j_2n_{k_2}+ \varepsilon_3j_3n_{k_3} + \varepsilon_4j_4n_{k_4} = 0]}
\\ \ll &\sum_{k_4 \in \Delta_i} \lvert c_{k_4} \rvert \sum_{\substack{k_4 - \lceil \log_q(3d)\rceil \leq k_3 \leq k_4\\k_3 \in \Delta_i}} \lvert c_{k_3} \rvert \sum_{\substack{k_2 \in \Delta_i\\k_2 \leq k_3}}|c_{k_2}|
\lvert c_{k_{1,k_3,k_4}(k_2)} \rvert \\
\ll &\sum_{k_4 \in \Delta_i} \lvert c_{k_4} \rvert \sum_{\substack{k_4 - \lceil \log_q(3d)\rceil \leq k_3 \leq k_4\\k_3 \in \Delta_i}}\lvert c_{k_3} \rvert \sum_{k_2 \in \Delta_i}c_{k_2}^2 \\
\ll &\left(\sum_{k \in \Delta_i}c_k^2\right)^2 \\
\ll & h(N)^\gamma \sum_{k \in \Delta_i} c_k^2.
\end{split}
\] 
This leads to $ \sum_{i=1}^M \mathbb{E} \left[ T_i^4 \right] \ll h(N)^{1+ \gamma} =o(h(N)^2) $ and consequently $ \sum_{i=1}^M \mathbb{E} \left[ Y_i^4 \right] =o(h(N)^2)$, as claimed.
\end{proof}

\subsection{An anisotropic CLT for trigonometric polynomials}\label{sec_key_prop}

We now establish the CLT in the case where $f = p$ is an even trigonometric polynomial of fixed degree $d \in \N$. Later, we will approximate $f \in \mathbb{L}_2(\T)$ by suitably chosen trigonometric polynomials. 
The following proposition establishes a CLT for $\sum_{k=1}^N c_k p(n_kx)$, by using the martingale CLT from Proposition~\ref{prop:martingale_clt}. 
\begin{prop}
\label{prop:clt_cosine_big_ck}
Let $(n_k)_{k \in \N}$ be a lacunary sequence satisfying the Hadamard-gap condition~\eqref{eq:Hadamard_gap} for some $q > 1$, let $p(x) = \sum_{j=1}^d a_j \cos(2 \pi j x) $ be a mean-zero even trigonometric polynomial of degree $d$ and let $(c_{k})_{k \leq N, N \in \N }$ satisfy Assumption~\ref{Ass:Condition_weights}. Then we have the following:
\begin{itemize}
    \item[(i)]
    If Assumption~\ref{Ass:dioph_cond_dream} is satisfied (with $p = f$ in~\eqref{non_deg_var}), then
\[
\frac{\sum_{k \leq N}c_{k} p(n_kx)}{\lVert \sum_{k \leq N}c_{k}p(n_kx) \rVert_2} \stackrel{w}{\longrightarrow } \mathcal{N}(0,1).
\]
 \item[(ii)]
   If Assumption~\ref{Ass:dioph_cond_homog} is satisfied, then 
\[
\frac{\sum_{k \leq N}c_{k} p(n_kx)}{\sqrt{h(N)} \lVert p \rVert_2} \stackrel{w}{\longrightarrow } \mathcal{N}(0,1).
\]
\end{itemize}
\end{prop}

\begin{proof}
We provide a complete proof of Proposition~\ref{prop:clt_cosine_big_ck} under Assumption~\ref{Ass:dioph_cond_dream}. The corresponding claim under Assumption~\ref{Ass:dioph_cond_homog} follows mutatis mutandis.
Let $N \in \N$ be given and recall that by Lemma~\ref{lem:Construction_blocks}, we can write
\[
\sum_{k =1}^{N}c_k p(n_k\cdot) = \sum_{i =1}^M Y_i + \sum_{i=1}^M (T_i-Y_i) + \sum_{k \in [1,N] \setminus (\Delta_1 \cup \ldots \cup \Delta_M)} c_k p(n_k \cdot ), 
\]
where $ \sum_{k \in [1,N] \setminus (\Delta_1 \cup \ldots \cup \Delta_M)} c_k^2 = o(h(N))$. We shall now establish two properties. On the one hand, we show that under Assumption~\ref{Ass:dioph_cond_dream}, the $2$-norms of the quantities 
\begin{equation}
\label{eq:residual_terms}
\sum_{i=1}^M (T_i- Y_i) \quad \text{and}\quad  \sum_{k \in [1,N] \setminus (\Delta_1 \cup \ldots \cup \Delta_M)} c_k p(n_k \cdot )
\end{equation}
are both of order $o(\lVert  S_N \rVert_2)$. On the other hand, we identify an asymptotically precise formula for $s_M$ as $N \longrightarrow \infty$, which will imply in particular that $s_M $ and   $\lVert S_N \rVert_2^2$  are asymptotically equivalent. Starting with the former, Lemma~\ref{lem:est_Ti_Yi} immediately implies  
\[
\left \lVert \sum_{i=1}^M (T_i- Y_i) \right \rVert_2 = O(1).
\]
On the other hand, using Parseval's identity we obtain   by Lemma~\ref{lem:Construction_blocks}(v)
\begin{align}\notag
\left \lVert \sum_{{k \in [1,N] \setminus (\Delta_1 \cup \ldots \cup \Delta_M)}} c_k p(n_kx) \right \rVert_2 
&\ll  \lVert p \rVert_\infty \sum_{1 \le j \le d} \left \lVert \sum_{k \in [1,N] \setminus (\Delta_1 \cup \ldots \cup \Delta_M)} c_k \cos(2 \pi j n_k x) \right \rVert_2 \\
& \ll \left(\sum_{k \in [1,N] \setminus (\Delta_1 \cup \ldots \cup \Delta_M)} c_k^2 \right)^{1/2} = o( \sqrt{h(N)}).\label{parseval}
\end{align} 
This means the residual terms from~\eqref{eq:residual_terms} have a negligible contribution on an $\mathbb{L}_2$ scale and we can focus on identifying the asymptotic order of $s_M$.
\par{}
By Lemma~\ref{lem:Construction_blocks}(iv) we have   $\# \Delta_i' \longrightarrow \infty$. This means, for fixed $d \in \N$, $i,i' \in \N$  and $N$  sufficiently large, and $k \in \Delta_i, k' \in \Delta_{i'}, i \neq i'$, we have
for all $1 \leq j,j' \leq d$ that $jn_k - j'n_{k'} \neq 0$. Using Lemma~\ref{lem:est_Ti_Yi} leads to
  \[
\begin{split}
\sum_{i=1}^M  \int_{\T} \left( \sum_{k \in \Delta_i} c_k p(n_k x) \right)^2 \,\mathrm{d}x 
&= \frac{1}{2} \sum_{i=1}^M \sum_{k,k' \in \Delta_i} \sum_{j,j' \le d} a_j a_{j'} c_k c_{k'} \mathbb{1}_{[j n_k - j' n_{k'} = 0]} \\
& = 
 \frac{1}{2} \sum_{\substack{k,k' \in \Delta_1 \cup \ldots \cup \Delta_M}} c_k c_{k'}  \sum_{j,j' \le d} a_j a_{j'} \mathbb{1}_{[j n_k - j' n_{k'} = 0]}  \\
& = \int_{\T} \left( \sum_{k \in \Delta_1 \cup \ldots \cup \Delta_M} c_k p(n_k x) \right)^2 \,\mathrm{d}x.
\end{split}
\]

By the triangle inequalities, we have
\[
\left\|
\sum_{k\in\Delta_1\cup\cdots\cup\Delta_M} c_k p(n_kx)
\right\|_2
\ge
\left\|\sum_{k\le N} c_k p(n_kx)\right\|_2
-
\left\|
\sum_{k\in[1,N]\setminus(\Delta_1\cup\cdots\cup\Delta_M)}
c_k p(n_kx)
\right\|_2
\] 
 and
\[
\left\|
\sum_{k\in\Delta_1\cup\cdots\cup\Delta_M} c_k p(n_kx)
\right\|_2
\le
\left\|\sum_{k\le N} c_k p(n_kx)\right\|_2
+
\left\|
\sum_{k\in[1,N]\setminus(\Delta_1\cup\cdots\cup\Delta_M)}
c_k p(n_kx)
\right\|_2.
\]
Therefore, using~\eqref{parseval},
\[
\left\|
\sum_{k\in\Delta_1\cup\cdots\cup\Delta_M} c_k p(n_kx)
\right\|_2
=
\left\|\sum_{k\le N} c_k p(n_kx)\right\|_2
(1+o(1)).
\]
 Thus, under Assumption~\ref{Ass:dioph_cond_dream} and Lemma~\ref{lem:est_Ti_Yi}, we have
\begin{equation}
\label{assum3_case}
\begin{split}
s_M & = \sum_{i=1}^M \int_{\T} \left( \sum_{k \in \Delta_i} c_k p(n_k x) \right)^2 \,\mathrm{d}x +O(1) \\
&= \int_{\T} \left( \sum_{k \leq N} c_k p(n_k x) \right)^2 \,\mathrm{d}x + o(h(N))
\\&= \left \lVert \sum_{k \leq N}c_{k}p(n_kx) \right \rVert_2^2(1 + o(1)),
\end{split}
\end{equation}
where we have used the estimate~\eqref{non_deg_var} from Assumption~\ref{Ass:dioph_cond_dream} in the last line.
This shows that $s_M$ and   $\lVert S_N \rVert_2^2$  are all asymptotically equivalent as $N \rightarrow \infty$, implying in particular that $h(N) \asymp s_M$. By Lemma~\ref{lem:est_VM_sM} and Proposition~\ref{prop:fourth_moment}, we thus have
\[
\frac{  \sum_{i=1}^M \E \left[Y_i^4 \right] + \E \left[ (V_M - s_M)^2\right]}{s_M^2} \longrightarrow 0, \quad \text{  as } M \rightarrow \infty,
\]
and consequently Proposition~\ref{prop:martingale_clt} implies
\[
\frac{Y_1 + \ldots + Y_M}{\sqrt{s_M}} \stackrel{w}{\longrightarrow} \mathcal{N}(0,1), \quad \text{ as } M \rightarrow \infty.
\]
  Combining  the estimates above, we have shown that under Assumption~\ref{Ass:dioph_cond_dream}
\[
\frac{\sum_{k \leq N}c_ k p(n_kx)}{\lVert \sum_{k \leq N}c_{k}p(n_kx) \rVert_2}= \frac{\sum_{i=1}^M Y_i}{\sqrt{s_M} } \cdot \frac{\sqrt{s_M}}{\lVert \sum_{k \leq N}c_{k}p(n_kx) \rVert_2} + r_N,
\]
where $r_N$ is a quantity with $\lVert r_N \rVert_2 \longrightarrow 0$. By~\eqref{assum3_case}, we have 
\[
\frac{\sqrt{s_M}}{\lVert \sum_{k \leq N}c_{k}p(n_kx) \rVert_2} \longrightarrow 1
\]
which completes the proof of (i) by virtue of Slutsky's Theorem.
\end{proof}

\subsection{The CLT in the general case}
\label{sec_final_CLT}
Having settled the CLT for lacunary sums generated by an even trigonometric polynomial, we now turn to Theorem~\ref{thm:iidclt}, whose proof is given below.
 The proof of  Theorem~\ref{thm:general_clt} follows by an analogous argument and will therefore be omitted.
\begin{proof}[Proof of Theorem~\ref{thm:iidclt}]
We start with the observation that Proposition~  \ref{prop:clt_cosine_big_ck}  can also be verified for a general trigonometric polynomial
\[
f(x) = \sum_{j=1}^d  a_j \cos(2 \pi j x) + \sum_{j=1}^d b_j \sin(2 \pi j x),
\]
without significant changes. 
\par{}
Now let us consider a general $f \in \mathbb{L}_2(\T)$
satisfying Assumption~\ref{Ass_f}. By renormalizing, we can assume without loss of generality that we can write
 \[
 f(x) = \sum_{j=1}^{\infty} \left[ a_j \cos(2 \pi j x) + b_j \sin(2 \pi j x) \right]
 \]
 with $ \max \{ |a_j|, |b_j| \} \leq \frac{1}{j^{\rho}}$ for all $j \in \N$. 
 Let \[\begin{split}p(x) &= \sum_{j=1}^d \left[ a_j \cos(2 \pi j x) + b_j \sin(2 \pi j x) \right],  \\r(x) &= \sum_{j=d+1}^{\infty} \left[ a_j \cos(2 \pi j x) + b_j \sin(2 \pi j x) \right],\end{split}\] and let us define
 \begin{align*}
 S_N  & = \sum_{k=1}^N c_k p(n_k x) + \sum_{k=1}^N c_k r(n_kx) \\
 & = S_N^p + S_N^r. 
\end{align*}
 We write $S_N^r = S_{N,e}^r + S_{N,o}^r$ where 
\[
S_{N,e}^r = \sum_{k \leq N}c_k\sum_{j=d+1}^{\infty} a_j \cos(2 \pi j n_k x), \quad 
S_{N,o}^r = \sum_{k \leq N}c_k\sum_{j=d+1}^{\infty}  b_j \sin(2 \pi j n_k x).
\]
By orthogonality, $\lVert S_N^r \rVert_2^2 = \lVert S_{N,e}^r \rVert_2^2 + \lVert S_{N,o}^r \rVert_2^2$.   Since  the estimate for $\lVert S_{N,o}^r \rVert_2 $ works completely analogously, we will only provide an estimate for $ \lVert S_{N,e}^r \rVert_2$. We observe that
  \[
\begin{split}
    \lVert S_{N,e}^r \rVert_2^2 & = \frac{1}{2}\sum_{1 \leq k,k'\leq N}c_kc_{k'}\sum_{j,j' > d}a_j a_{j'}\mathds{1}_{[jn_k - j'n_{k'} = 0]} \\
    &= \frac{1}{2}\sum_{1\leq k\leq N}c_k^2\sum_{j > d}a_j^2 + \sum_{1 \leq k < k' \leq N}c_kc_{k'}\sum_{j,j' > d}a_ja_{j'}\mathds{1}_{[jn_k - j'n_{k'} = 0]}
    \\&\ll \sum_{j > d}\frac{1}{j^{2\rho}} h(N) + \sum_{1 \leq k < k' \leq N} \lvert c_kc_{k'} \rvert \sum_{j > d}|a_j|\sum_{j' > d} |a_{j'}|\mathds{1}_{[jn_k - j'n_{k'} = 0]},
\end{split}
\] 
where the last estimate employs $|a_j| \leq j^{- \rho}$. Note that 
$jn_k - j'n_{k'} = 0$ and thus 
  $j =  \frac{j'n_{k'}}{n_{k}} \geq j'q^{k'-k}$. This means we obtain (with the formal definition of $c_k = 0$ for $k > N$) that
  \[
\begin{split}
\sum_{1 \leq k < k' \leq N} \lvert c_kc_{k'} \rvert \sum_{j > d}|a_j|\sum_{j' > d} |a_{j'}|\mathds{1}_{[jn_k - j'n_{k'} = 0]}
&\leq \sum_{1 \leq k < k' \leq N} \lvert c_kc_{k'} \rvert \sum_{j' > d}\frac{1}{j'^{\rho}}\frac{1}{j'^{\rho}}q^{\rho(k-k')}\\
& \le \sum_{j' > d}\frac{1}{j'^{2\rho}}\sum_{1 \leq k \leq N} \lvert c_k \rvert  \sum_{e \leq N} \lvert c_{k+e} \rvert q^{-\rho e}\\
&= \sum_{j' > d}\frac{1}{j'^{2\rho}} \sum_{e \leq N}q^{-\rho e}\sum_{1 \leq k \leq N} \lvert c_kc_{k+e} \rvert \\
& \ll \sum_{j' > d}\frac{1}{j'^{2\rho}} \sum_{e \leq N}q^{-\rho e}\sum_{1 \leq k \leq N}c_k^2
\\&\ll h(N)\sum_{j' > d}\frac{1}{j'^{2\rho}}.
\end{split}
\] 
Since $\rho > \frac{1}{2}$, for any given $\varepsilon > 0$, we can choose $d \in \N$ sufficiently large such that   $C \sum_{j' > d}\frac{1}{j'^{2\rho}} \leq \varepsilon$ with $C$ the implied constant arising in the above. This shows  that $\lVert S_{N}^r \rVert_2^2 \leq \varepsilon h(N)$.\\
We now define
\[
\sigma_N^f = \sqrt{h(N)} \lVert f \rVert_2, \quad \sigma_N^p = \sqrt{h(N)} \lVert p \rVert_2.
\]
Then by construction, we have
\[
\left (1- \varepsilon \right) \sigma_N^p \leq \sigma_N^f \leq \left (1+ \varepsilon \right) \sigma_N^p,
\]
for any given $\varepsilon > 0$, provided the degree $d \in \N$ of $p$ is sufficiently large.
 We now get, using the elementary estimate $\mathbb{P}[A \cap B] \geq \mathbb{P}[A] -\mathbb{P}[B^c]$ and Chebyshev's inequality,  for $t \geq 0$,
  \begin{align*}
     \mathbb{P} \left[ \frac{S_N}{\sigma_N^f} \geq t \right]
     & \geq \mathbb{P} \left[ \frac{S_N^p}{\sigma_N^f} \geq t + \frac{\lvert S_N^r \rvert}{\sigma_N^f}, \frac{\lvert S_N^r \rvert}{\sigma_N^f} \le \sqrt[4]{\varepsilon} \right]\\
     & \geq \mathbb{P} \left[ \frac{S_N^p}{\sigma_N^f} \geq t + \sqrt[4]{\varepsilon} \right] - 
     \mathbb{P} \left[ \frac{|S_N^r|}{\sigma_N^f} \geq \sqrt[4]{\varepsilon}\right] \\
     & \geq \mathbb{P} \left[ \frac{S_N^p}{\sigma_N^f} \geq t + \sqrt[4]{\varepsilon} \right] - \frac{ \lVert S_N^r \rVert_2^2}{\left(\sigma_N^f\right)^2 \sqrt{\varepsilon}} \\
      & \geq \mathbb{P} \left[ \frac{S_N^p}{\left(1+ \varepsilon \right) \sigma_N^p} \geq t + \sqrt[4]{\varepsilon} \right] - \frac{ \sqrt{\varepsilon} }{ \lVert f \rVert_2^2 } ,
 \end{align*}
 where, in the last step, we have used that $ \lVert S_N^r \rVert_2^2 \leq \varepsilon h(N)$ and $\sigma_N^f = \lVert f \rVert_2 \sqrt{h(N)}$. If we take the limit inferior as $N \longrightarrow \infty$ on both sides, by Proposition~\ref{prop:clt_cosine_big_ck} we thus obtain 
 \[
 \liminf_{N \rightarrow \infty}\, \mathbb{P} \left[ \frac{S_N}{\sigma_N^f} \geq t \right] \geq 1-\Phi \left( (1 + \varepsilon)(t+ \sqrt[4]{\varepsilon}) \right) -\frac{ \sqrt{\varepsilon} }{ \lVert f \rVert_2^2 } ,
 \]
 where $\Phi$ denotes the standard normal distribution function. Now we can take the limit as $\varepsilon \longrightarrow 0$ to get
 \[
 \liminf_{N \rightarrow \infty} \,\mathbb{P} \left[ \frac{S_N}{\sigma_N^f} \geq t \right] \geq 1- \Phi (t).
 \]
 The converse, i.e., 
 \[
 \limsup_{N \rightarrow \infty} \mathbb{P} \left[ \frac{S_N}{\sigma_N^f} \geq t \right] \leq 1- \Phi (t)
 \]
 can be shown analogously. Analogous arguments can be applied for the case where $t< 0$. 
 In total, we have shown that 
 \[
 \frac{S_N}{\sigma_N^f} \stackrel{w}{\longrightarrow} \mathcal{N}(0,1), \quad \text{as} \quad N \rightarrow \infty. \qedhere
 \]
\end{proof}

\bibliographystyle{abbrv}
\bibliography{bibliography}

\vspace{1cm}
       {\small
\noindent Lorenz Fr\"uhwirth:\\
\\
 {Faculty of Computer Science and Mathematics, University of Passau, Dr.-Hans-Kapfinger-Strasse 30,}\\
 Passau, Germany\\
 E-mail: {\tt lorenz.fruehwirth@uni-passau.de}

\bigskip

{\small
\noindent Manuel Hauke-Treuer:\\
\\
Department of Mathematical Sciences,
NTNU Trondheim,
Sentralbygg 2 Gløshaugen,\\Trondheim, Norway\\
\\
Institute of Analysis and Number Theory, TU Graz,
Steyrergasse 30,\\Graz, Austria\\
 E-mail: {\tt hauke@math.tugraz.at}

\end{document}